\documentclass{amsart}

\usepackage{etex}
\usepackage{amsmath, amssymb}
\usepackage{array}
\usepackage[frame,cmtip,arrow,matrix,line,graph,curve]{xy}
\usepackage{graphpap, color, paralist}
\usepackage[mathscr]{eucal}
\usepackage[pdftex]{graphicx}
\usepackage[pdftex,colorlinks,backref=page,citecolor=blue]{hyperref}
\usepackage{pifont}
\usepackage{cancel}
\usepackage{mathtools}
\usepackage{tikz}

\usepackage{bbm}
\usepackage{dsfont}
\newcommand{\DD}{\mathbb{D}}

\newcommand{\pr}{\mathbb{P}}

\newcommand{\E}[0]{\mathbb{E}}

\newcommand{\beq}[1]{\begin{equation}\label{#1}}
\newcommand{\enq}[0]{\end{equation}}

\newcommand{\mn}[0]{\medskip\noindent}
\newcommand{\nin}[0]{\noindent}

\newcommand{\sub}[0]{\subseteq}
\newcommand{\sm}[0]{\setminus}
\renewcommand{\dots}[0]{,\ldots,}

\newcommand{\ov}[0]{\overline}

\newcommand{\less}[0]{~~\mbox{\raisebox{-.6ex}{$\stackrel{\textstyle{<}}{\sim}$}}~~}
\newcommand{\more}[0]{~~\mbox{\raisebox{-.9ex}{$\stackrel{\textstyle{>}}{\sim}$}}~~}

\newcommand{\pee}[0]{{\mathcal P}}

\newcommand{\U}[0]{{\mathcal U}}

\newcommand{\ra}[0]{\rightarrow}
\newcommand{\Ra}[0]{\Rightarrow}

\newcommand{\TTT}[0]{T}

\newcommand{\supp}[0]{{\rm supp}}

\newcommand{\sfa}[0]{{\sf a}}
\newcommand{\sfb}[0]{{\sf b}}
\newcommand{\sfc}[0]{{\sf c}}

\newcommand{\ttt}[0]{t}

\newcommand{\nv}[0]{\nabla_v}

\newcommand{\0}[0]{\emptyset}

\newcommand{\C}[2]{{{#1}\choose{{#2}}}}
\newcommand{\Cc}[0]{\tbinom}
\newcommand{\ga}[0]{\alpha }
\newcommand{\gb}[0]{\beta }
\newcommand{\gc}[0]{\gamma }
\newcommand{\gd}[0]{\delta }
\newcommand{\gD}[0]{\Delta }
\newcommand{\gG}[0]{\Gamma }

\newcommand{\gl}[0]{\lambda }

\newcommand{\go}[0]{\omega}
\newcommand{\gO}[0]{\Omega}

\newcommand{\gs}[0]{\sigma}
\newcommand{\gS}[0]{\Sigma}
\newcommand{\gz}[0]{\zeta}
\newcommand{\gzz}[0]{\zeta}
\newcommand{\eps}[0]{\varepsilon }
\newcommand{\vt}[0]{\vartheta}
\newcommand{\vs}[0]{\varsigma}
\newcommand{\vr}[0]{\varrho}

\newcommand{\III}[0]{I}
\newcommand{\JJJ}[0]{J}

\newcommand{\prh}[1][]{\pr_h}
\newcommand{\oH}[1][]{\ov{H}}
\newcommand{\oB}[1][]{\ov{B}}

\newtheorem{thm}{Theorem}[section]
\newtheorem{prop}[thm]{Proposition}

\newtheorem{lemma}[thm]{Lemma}
\newtheorem{obs}[thm]{Observation}
\newtheorem{cor}[thm]{Corollary}

\newtheorem{conj}[thm]{Conjecture}
\theoremstyle{definition}

\newtheorem{claim}[thm]{Claim}

\newcommand{\mbone}[0]{\mathbbm{1}}

\newcommand{\QQQ}[0]{A}
\newcommand{\kkk}[0]{k}
\newcommand{\uuu}[0]{u}

\newcommand{\NNN}[0]{W}
\newcommand{\LLL}[0]{T}
\newcommand{\UUU}[0]{U}
\newcommand{\VVV}[0]{V}
\newcommand{\vvv}[0]{v}
\newcommand{\www}[0]{w}
\newcommand{\WWW}[0]{W}
\newcommand{\YYY}[0]{Z}
\newcommand{\ZZZ}[0]{Z}
\newcommand{\gSR}[0]{R}

\newcommand{\cha}[1]{\textcolor{teal}{#1}}

\newcommand{\bbb}[0]{c}

\begin{document}

\title{Asymptotics for palette sparsification from variable lists}

\author{Jeff Kahn \and Charles Kenney}
\thanks{Department of Mathematics, Rutgers University}
\thanks{JK was supported by NSF Grant DMS1954035 and by a grant from the Simons Foundation (915224, Kahn)}
\email{jkahn@math.rutgers.edu,ctk47@math.rutgers.edu}
\address{Department of Mathematics, Rutgers University \\
Hill Center for the Mathematical Sciences \\
110 Frelinghuysen Rd.\\
Piscataway, NJ 08854-8019, USA}

\begin{abstract}

It is shown that the following holds for each $\eps >0$.
For $G$ an $n$-vertex graph of maximum degree $D$, lists $S_v$
of size $D+1$ (for $v\in V(G)$),
and $L_v$ chosen uniformly from 
the ($(1+\eps)\ln n$)-subsets of $S_v$ (independent of other choices),
\[
\mbox{$G$ admits a proper coloring $\gs$ with $\gs_v\in L_v$ $\forall v$}
\]
with probability tending to 1 as $D\ra \infty$.

When each $S_v $ is $\{1\dots D+1\}$, this 
is an asymptotically optimal version of the ``palette sparsification'' theorem of 
Assadi, Chen and Khanna that was proved in an earlier paper by the present authors.

\end{abstract}

\maketitle
\section{Introduction}\label{Intro}

Recall that a \emph{list coloring} of a graph $G$ (on vertex set $V$)
is a proper vertex coloring in which each vertex receives a color from some
prescribed list (particular to that vertex).  Formally, 
for $L=(L_v:v\in V)$ with $L_v\sub \gG = \{\mbox{``colors''}\}$,
an \emph{L-coloring} of $G$ is a (proper) coloring $\gs$ of $V$ with
$\gs_v\in L_v$ $\forall v$.  If such exists then we also say $G$ is \emph{L-colorable}.
(See \cite{Vizing,ERT} for origins and e.g.\ \cite{Diestel} for general list coloring background.)

In this paper we are interested in
the situation
where $\gG$ is arbitrary and, for some $\ell$,
each $L_v$ is a uniform $\ell$-subset of a given $S_v\sub \gG$.
(So we assume $|S_v|\geq \ell ~\forall v$.
In our context this model was first considered by Alon and Assadi \cite{AA}.)
The following, our main result, was Conjecture~10.3 in \cite{APS}
(where we said we \emph{believed} we could prove it).
Here and throughout the paper we use $n=|V|$ and $\log=\ln$.
\begin{thm}\label{CKK}
For fixed $\gd>0$, 
if 
$\gD_G \leq D$, $|S_v|=D+1$ $\forall v$, and $\ell =(1+\gd)\log n$, 
then $G$ is L-colorable w.h.p.\footnote{\emph{with high probability}, meaning with probability 
tending to 1 as $n\ra \infty$.}
\end{thm}
\nin
Equivalently, 
the theorem holds with $\gd$ a slow enough $o(1)$;
in fact $\gd =\go(1/\log n)$ should 
suffice, but we don't show this.

Theorem~\ref{CKK} with $\gG= S_v=[D+1]:=\{1\dots D+1\}$ $\forall v$ was 
suggested by Assadi and proved in
\cite{APS}, motivated by 
the following ``palette sparsification'' 
theorem of Assadi, Chen and Khanna~\cite{ACK}.
(See \cite{ACK}, as well as \cite{AA,HFNT}, 
for the surprising algorithmic consequences.)
\begin{thm}\label{TACK}
There is a fixed C such that if $\gD_G = D$ and the $L_v$'s are drawn uniformly and independently
from the ($C\log n$)-subsets of
$\gG=[D+1]$, then G is L-colorable w.h.p.
\end{thm}

As mentioned above, the more general (and from a list-coloring standpoint natural)
case of variable lists was introduced by
Alon and Assadi \cite{AA} (see also \cite{HFNT}), who showed
(with $G$ again on $V=[n]$ and $d_v:=d_G(v)$):
\begin{thm}\label{TAA}
In each of the following situations $G$ is L-colorable w.h.p.

\nin
{\rm (a)}  $|S_v|=(1+\eps)d_v$ ($\forall v$) and $\ell = C_\eps\log n$ 
(for any fixed $\eps >0$ and a suitable $C_\eps$).

\nin
{\rm (b)}  $S_v=[d_v+1]$ ($\forall v$) and $\ell = C\log n$ 
(for a suitable $C$).
\end{thm}

We suspect that Theorem~\ref{CKK} extends to the ``correct'' version of Theorem~\ref{TAA}:

\begin{conj}\label{CKK2}
For fixed $\eps >0$,
if $|S_v|=d_v+1$ $\forall v\in V$ and $\ell = (1+\eps)\log n$,
then G is L-colorable w.h.p.
\end{conj}

The proof of Theorem~\ref{CKK} is previewed in Section~\ref{Sketch}, 
at the end of which we will be in 
a better position to say something intelligible about how the paper is organized.
As will be noted at various points below, the argument here---and some of the prose---has 
considerable overlap with 
\cite{APS}, to which we will appeal where possible to avoid duplication.
See also Section~10 of \cite{APS}
for discussion of a few related problems.

\section{Sketch}\label{Sketch}

Here we give a quick indication of the proof of Theorem~\ref{CKK}. 
%noting that while the strategies described are reasonably natural, 
%their implementation is delicate and not very straightforward (and that most of what's 
%sketched here will be easier to say comprehensibly in the more formal settings to come).
We will
assume $G$ is $D$-regular, as we may
by the next observation, which was \cite[Prop.~2.1]{APS}.
\begin{prop}\label{Preg}
Any $n$-vertex simple graph $G$ with $\gD_G\leq D$ is contained in a $D$-regular
simple graph with at most $n+D+2$ vertices.
\end{prop}

\nin
So we are given $\gd>0$ and want to show that for 
any $D$-regular $G$ and $L=(L_v:v\in V)$
as in Theorem~\ref{CKK},
$G$ is $L$-colorable w.h.p.

As in \cite{ACK} and \cite{APS}, our argument begins with a partition of 
$G$ into a ``sparse'' part, $V^*$, and
small, dense ``clusters.''
The following statement from \cite{APS} is contained in \cite[Lemma 2.3]{ACK},
a variant of a decomposition from \cite{HSS}.
\begin{lemma}\label{LACK}
For any D, $\eps>0$, and $D$-regular graph $G$ on vertex set $V$, 
there is a partition
$
V=V^*\cup C_1\cdots\cup C_m
$
such that:

\nin
{\rm (a)}  
$~v\in V^* ~\Ra ~ |\{w\sim v: d(v,w) < (1-\eps)D\}| > \eps D$;

\nin
{\rm (b)} 
$\,$ for all $i\in [m]$:  

{\rm (i)} \, $|C_i|\in [(1-\eps)D, (1+6\eps)D]$, and

{\rm (ii)}\, for all $v\in C_i$, $|N_v\sm C_i|< 7\eps D$ and $|C_i\sm N_v|< 6\eps D$.
\end{lemma}
\nin
(The specific constants are unimportant.)
An easy consequence of (a)---and all we will use it for---is
\beq{generic'} 
|\ov{G}[N_v]| > (\eps^2/2)D^2\,\,\,\,\forall v\in V^*.
\enq

We will use Lemma~\ref{LACK} with $\eps$ a small constant (depending on $\gd$),
to which we will not assign a specific value; constraints on what that value might be
will appear mostly in the ``dense'' phase of the argument, beginning in Section~\ref{Clusters}
(see ``Note on parameters'' near the beginning of that section).

\mn

As in \cite{APS} (and \cite{ACK}) we color in two stages, the first (``sparse'') coloring $V^*$ and the second 
(``dense'') extending to the $C_i$'s.

\nin
\textbf{Sparse phase.}  Our aim here---as in \cite{APS}, but no longer \cite{ACK}---is 
to reduce to the following result of 
Reed and Sudakov \cite{RS}.
For  $v\in V$ and $\gc\in \gG$, define the \emph{color degree}
\[
d_\gc(v):=|\{w\sim v:\gc\in L_w\}|.
\]
\begin{thm}\label{TRS}
For any $\vr>0$ there is a $t_\vr$ such that
if $t> t_\vr$, 
all color degrees are at most $t$, and all list sizes are at least $(1+\vr)t$, 
then $G$ is $L$-colorable.
\end{thm}

\nin

We will use Theorem~\ref{TRS} with lists that are subsets of the $L_v$'s of Theorem~\ref{CKK}.
Our initial plan here (which doesn't always work) is as follows.
We first tentatively assign single random colors to the vertices 
(the color assigned to $v$ being a first installment on $L_v$), retaining these 
colors at the properly colored vertices
(those assigned colors not assigned to their neighbors) and discarding the rest. 
We then choose the balances of the lists and, for $v$ not colored in the first stage,
confine ourselves to the surviving colors in $L_v$ (those not already used at neighbors of $v$).
This is supposed to work as follows.

For a given $v$ the first stage above should yield:

\nin
(a)  a positive fraction (essentially $e^{-1}$) of the neighbors of $v$ properly colored;

\nin
(b)  the number of colors \emph{from} $S_v$ assigned
to the neighbors in (a) is significantly less 
than the number of these neighbors, either because of
significant duplication among the colors 
(a benefit of being in the sparse phase),
or because many of the chosen colors are outside $S_v$.

\nin
It is not hard to see that if (a) and (b) (suitably quantified) hold for all $v$, then 
the \emph{natural} values (with respect to still-allowed colors)
of the list sizes and color degrees support use of Theorem~\ref{TRS}.

But of course we need to say that w.h.p.\ all 
(or a strong enough \emph{most}) of these
desirable behaviors 
%(i.e.\ (a), (b) and the ``natural'' values above) 
happen \emph{simultaneously}.
This turns out to be true (or true \emph{enough})
for $n$ less than some quite large function of $D$,
but---surprisingly---not all the way to the maximum $n$, roughly $e^D$, for which $\ell < D+1$.
Discussion of the resolution of this issue is postponed to Section~\ref{Sparse}.

\nin

\nin
\textbf{Dense phase.}
We process the $C_i$'s one at a time, in each case assuming an arbitrary coloring
(even an arbitrary \emph{labeling} would be okay here) of $V^*$ and the earlier clusters.

With $C_i=C$, the object of interest here is the natural bigraph on $C\cup \gG$ 
that connects each $v\in C$ 
to the colors in $S_v$ not yet assigned to neighbors of $v$ outside of $C$,
the obvious goal being to show that the random restriction to edges $(v,\gc)$ with 
$\gc\in L_v$ is likely to admit a $C$-perfect matching.  
(E.g.\ if $G$ consists of disjoint $K_{D+1}$'s then we are asking for likely existence of
such matchings in $n/(D+1)$ copies of $K_{D+1,D+1}$, and standard isolated vertex 
considerations---though the random subgraphs here are not quite the usual ones---show that the
$\ell$ of Theorem~\ref{CKK} is asymptotically best possible.)
But in general we will have situations where such matchings clearly (or not so clearly) 
cannot---or with not so small probability \emph{may} not---exist;
and even when they \emph{should} exist, showing that they (probably) do has---to date---proved
surprisingly far from straightforward.

We will not try to sketch this, but give a quick, vague hint.
There are eventually three main ``regimes,'' depending (\emph{very} roughly)
on how the $S_v$'s 
overlap.  
(Which regime we are in depends on whether \eqref{R1} holds, and if 
it does not then whether \eqref{Ssize} does.)
The most difficult of these, which generalizes the main regime of \cite{APS},
is the one for which we don't expect full matchings;  e.g., most naively, when $|C|> D+1$ 
and the $S_v$'s are all the same.   
Here the natural (but in practice delicate) plan is to
precolor enough of $C$ to reach a point where a good matching of what 
remains \emph{is} likely.

As suggested above, the other two regimes, where we show a good matching \emph{is} likely,
look easy but don't seem to be.
These have no counterparts in \cite{APS} and seem to be the most interesting 
parts of what this work adds to what we knew earlier.

\mn

After reviewing a few standard tools in Section~\ref{Prelim},
we give the basic argument for the sparse phase in Section~\ref{Sparse}
postponing the proofs of Lemmas~\ref{LPh1a} and \ref{LC0}, 
which carry most of the load.
The martingale concentration machinery needed for these is reviewed in Section~\ref{Martingales},
and the lemmas themselves are proved in Section~\ref{Concentration}.
The dense phase is treated in Sections~\ref{Clusters}-\ref{Matchings}, again
with the main line of argument in Section~\ref{Clusters}, followed by proofs
of the supporting 
Lemmas~\ref{LMg'A} and \ref{Lsucc} (in reverse order) 
in Sections~\ref{Matchings} and \ref{Process}.
(The sparse and dense parts of the argument can be read independently.)

\nin
\textbf{Usage}

The vertex set $V=V(G)$ is of size $n$, and $\gG = \cup_{v \in V} S_v$.
Any $\gc$, perhaps subscripted, is by default a member of $\gG$.
In graphs, $\nabla(X,Y)$ is the set of edges joining the disjoint vertex sets $X,Y$
and $\nabla(X)=\nabla(X,V\sm X)$; $N(\cdot)$ is neighborhood
and $N(S)=\cup_{v\in S}N(v)$; and $H[X]$ is the subgraph of $H$ induced by $X$.

Asymptotics:  we use $o(\cdot)$, $\go(\cdot)$, $\sim$, etc.\ in the usual ways, and often write 
$a\less b$ for $a<(1+o(1))b$  (all as $D\ra \infty$).
Our discussion, especially in the dense phase, will involve a number of constants,
but the implied constants in $O(\cdot)$, $\gO(\cdot)$, $\Theta(\cdot)$
do not depend on these.

We will always assume $D$ is large enough to support
our assertions and, following a common abuse, usually pretend 
large numbers are integers. In what follows, we will more often be interested in $D+1$ (the common size of
the $S_v$'s) than in $D$ itself;
so to avoid cluttering the paper with $(D+1)$'s we proceed as follows.
For the sparse phase, the difference is irrelevant 
(the argument doesn't really
change if we shrink $S_v$'s by 1), 
and we ignore it, using $D$ for both $D$ and $D+1$
(in these cases even somewhat looser 
interpretations of $D$ wouldn't change anything);
but when we come to the dense phase, the difference does sometimes matter, 
so in Sections~\ref{Clusters}-\ref{Matchings} we will---but only when it \emph{does} matter---use
$\DD$ for $D+1$.

We will sometimes use square brackets to indicate quick comments and easy justifications.

\nin

\section{Preliminaries}\label{Prelim}

We superfluously recall the Lov\'asz Local Lemma (\cite{ErLov} or e.g.\
\cite[Ch.\ 5]{AS}):
\begin{lemma}\label{LLL}
Suppose $A_1,\ldots, A_m$ are events in a probability space with
$\pr(A_i)\leq p$ $\forall i$,
and $\frak{G}$ is a graph on $[m]$
with $A_i$ independent of $\{A_j: j\not\sim_{\frak{G}} i\}$ ($\forall i$). 
If $ep\cdot (\gD_{\frak{G}} +1)<1$, then
$
\pr(\cap \bar{A_i}) >0.
$
\end{lemma}

Most of what we do below relies on various
concentration assertions, beginning with the following
standard
``Chernoff"
bounds.
Recall that a random variable $\xi$ is \emph{hypergeometric} if, for some
$s,a$ and $k$,
it is distributed as $|X\cap A|$, where $A$ is a fixed $a$-subset
of the $s$-set $S$ and $X$ is uniform from $\C{S}{k}$.
\begin{thm}
\label{T2.1}
If $\xi $ is binomial or hypergeometric with  $\E \xi  = \mu $, then for $t \geq 0$,
\begin{align}
\pr(\xi  \geq \mu + t) &\leq
\exp\left[-\mu\varphi(t/\mu)\right] \leq
\exp\left[-t^2/(2(\mu+t/3))\right], \label{eq:ChernoffUpper}\\
\pr(\xi  \leq \mu - t) &\leq
\exp[-\mu\varphi(-t/\mu)] \leq
\exp[-t^2/(2\mu)],\label{eq:ChernoffLower}
\end{align}
where $\varphi(x) = (1+x)\log(1+x)-x$
for $ x > -1$ and $ \varphi(-1)=1$.
\end{thm}
\nin
(See e.g.\ \cite[Theorems 2.1 and 2.10]{JLR}.)
For larger deviations the following consequence of the finer bound in \eqref{eq:ChernoffUpper}
is helpful.  
\begin{thm}
\label{Cher'}
For $\xi $ and $\mu$ as in Theorem~\ref{T2.1} and any $K$,
\beq{largedev}
\pr(\xi  > K\mu) < \exp[-K\mu \log (K/e)].
\enq
\end{thm}
\nin

We also need the simplest case of the Janson Inequalities
(see \cite{Janson} or \cite{JLR}, in particular
\cite[Theorem 2.18(ii)]{JLR} for the case in question).  
Suppose $A_1\dots A_m$ are subsets of a
finite set $S$; let $E_i$ be the event $\{S_p\supseteq A_i\}$
(with $S_p$ the usual ``$p$-random'' subset of $S$);
and set $\mu =\sum\pr(E_j)$ and
\[   
\ov{\gD} = \sum\sum\{\pr(E_iE_j): A_i\cap A_j\neq\0\}.
\]   
(Note this includes
diagonal terms.)

\begin{thm}\label{TJanson}
With notation as above, 
$\,\,\pr(\cap \ov{E}_i) \leq \exp[-\mu^2/\ov{\gD}]$.
\end{thm}

\begin{prop}\label{Pcvx}
Let $0 \leq b \leq a$, and for a random variable $X$ 
taking finitely many values in $[0,a],$
let
\[
\E'[X] = \E[X \mbone_{\{X \leq b\}}].
\]
Suppose $X$ is such a random variable with 
\beq{EXE'X}
\mbox{$\E[X] \geq \alpha\,\,$ and $\,\,\E'[X] \geq \beta$,}
\enq
where $0\leq \gb\leq \ga$.
Then 
the r.v.\ $\YYY$ with
\[
\mbox{$\pr(\YYY=b) = \gb/b$, $\pr(\YYY=a) = (\ga-\gb)/a~$ and 
$~\pr(\YYY=0) = 1-(\pr(\YYY=b)+\pr(\YYY=a))$}
\]
satisfies \eqref{EXE'X} (with equality)
and, for any convex $g:[0,a]\ra \mathbb{R}$
with $g(a) \leq g(0),$ 
\[
\E[g(\YYY)]  \geq \E[g(X)].
\]
\end{prop}
\begin{proof}
We use $X\sim p$ for $\pr(X=x)=p(x)$ and $\supp(X) =\{x:p(x)\neq 0\}$.
It will be convenient to assume 
\beq{gassume}
\mbox{$g$ is strictly convex with $g(0)> g(a)$}
\enq
(as we may since any $g$ as in the proposition is a limit of $g$'s satisfying \eqref{gassume}).
It is then enough to show

\nin
\emph{Claim.}
If $Y$ satisfies \eqref{EXE'X} and $\E [g(Y)]$ is maximum subject to 
\eqref{EXE'X} and
\[
\supp(Y)\sub \supp(X) \cup \{0,b,a\},
\]
then $Y=Z$.

\nin
(For if $Y$ is as in the claim---note there \emph{is} such a $Y$---then 
$\E [g(Z)] = \E [g(Y)]\geq \E [g(X)]$.)

\mn
\emph{Proof of Claim.}  Supposing $Y\sim p$, say 
$W$ is obtained from $Y$ by \emph{spreading weight $\gc$ from $x$ to $\{y,z\}$} if
%$y,z\in [0,a]$,
$x=\gl y+(1-\gl)z$ for some $\gl\in (0,1)$,
$\gc\in (0,p(x)]$, and $W\sim q$ with
\[
q(w) =\left\{\begin{array}{ll}
p(y) +\gl \gc&\mbox{if $w=y$,}\\
p(z) +(1-\gl )\gc&\mbox{if $w=z$,}\\
p(x) -\gc&\mbox{if $w=x$,}\\
p(w)&\mbox{otherwise;}
\end{array}\right.
\]
and notice that in this case $\E W = \E Y$,
\[
\mbox{$\E' W = \E'Y\,$ if $\{x,y,z\}\sub [0,b]$ or $\{x,y,z\}\sub (b,a]$,}
\]
and (by strict convexity of $g$) $\E[g(W)]> \E[g(Y)]$.

The claim is then given by the following consequences of our assumptions on $Y$.

\nin
(a)  $\supp (Y)\sub \{0,b,a\}$.
(For if $x\in \supp(Y)\sm \{0,b,a\}$ then spreading weight $p(x)$ from $x$ to $\{0,b\}$
or $\{b,a\}$ (depending on whether $x$ is in $(0,b)$ or $(b,a)$), increases $\E g$
without affecting \eqref{EXE'X}, contrary to assumption.)

\nin
(b)  $E' Y \,\,( = p(b)\cdot b) =\gb$.
(Or spreading weight $p(b)-\gb/b$ from $b$ to $\{0,a\}$ produces $W$ with 
$\E W =\E Y$, $\E' W=\gb$, and $\E g(W) > \E g(Y)$,
again a contradiction.)

\nin
(c)  $Y=Z$.
(Or $p(a) >  \pr(Z=a)$ and $\E g(Y)-\E g(Z) = (p(a)-\pr(Z=a)) (g(a)-g(0)) <0$,
a final contradiction.)

\end{proof}

\section{Coloring $V^*$}\label{Sparse}

For the rest of the paper we assume the decomposition $V = V^* \cup C_1 \cup \cdots \cup C_k$
of Lemma~\ref{LACK}, and, setting $\vt = \eps^2/2$, recall
\beq{generic'}
|\overline{G}[N_v]| > \vt D^2 \ \ \ \forall v \in V^*,
\enq
which is all we need from the sparsity of $V^*.$

The goal of the present section is to produce an $L$-coloring of $V^*$.
Here it will be convenient to retain the cluster vertices, assigning them 
temporary ``dummy'' lists that we will later discard.

Fix $\eps$ small enough (depending on $\gd$) to support what follows.
(As noted following Lemma~\ref{LACK}, what this entails will appear mainly in 
Section~\ref{Clusters}; again, see ``Note on parameters'' following \eqref{notext}.)

We first want to say that we can (w.h.p.) specify, at some affordably small cost in list size, $T \subseteq V$
and an $L$-coloring $\sigma$ of $T$ such that for every $v \in V^*$,
\beq{ideal1}
|T \cap N_v| \sim e^{-1} D
\enq
and
\beq{ideal2}
|T \cap N_v| - |\sigma(T \cap N_v) \cap S_v| > \vt' D,
\enq
where
\beq{defVtPrime}
\vt' = e^{-3} \vt /4\ \ \ (= e^{-3} \eps^2/8).
\enq
(This value corresponds to the bound in \eqref{expectManyFraternalEvents},
but any fixed, positive $\vt'$ would do as well.
Membership of $v$ in $V^*$ is irrelevant for \eqref{ideal1}
but is, obviously,
needed for \eqref{ideal2}.)
Note that the combination of \eqref{ideal1} and \eqref{ideal2} gives
\[   %\beq{ideal3}
|S_v\sm \gs(T\cap N_v)|\more (1-e^{-1}+\vt')D.
\]   %\enq

Given $T$, $\gs$ satisfying \eqref{ideal1} and \eqref{ideal2},
the reduction to Theorem~\ref{TRS} is exactly as in \cite{APS} 
(making no use of the equality of the $S_v$'s) and we will not repeat it,
but briefly recall what ought to happen.
We choose fresh lists $L_v'$ of size $t$ slightly smaller than $\ell$ (to allow for colors
used in producing $(T,\gs)$) and let 
\[
L_v'':=L_v'\sm \gs(T\cap N_v)
\]
(the usable part of $L_v'$).  We then have the natural behaviors
\beq{nat1}
 |L_v''| \approx 
|S_v\sm \gs(T\cap N_v)|t/(D+1)\more (1-e^{-1}+\vt')t
\enq
and (for $\gc\in S_v$, with $d_\gc''$ denoting color degree w.r.t.\ $L''$),
\beq{nat2}
d''_\gc(v) \approx |\{w\sim v:\gc\in S_w\}| t/(D+1)
\less (1-e^{-1} )t
\enq
(which can badly overestimate if the $S_w$'s for $w\sim v$ aren't 
much like $S_v$). 

When \emph{actual} behavior agrees with \eqref{nat1} and 
\eqref{nat2}, Theorem~\ref{TRS} immediately
gives the desired extension of $\gs$ to $V^*\sm T$; but in reality there will 
will usually be exceptions, dealing with which takes some effort.
But, again, we 
refer to \cite{APS} for details and turn to the more interesting
question of arranging \eqref{ideal1} and \eqref{ideal2}.

\mn

A natural way to try to accomplish this
%\eqref{ideal1} and \eqref{ideal2} 
is as follows.  
Independently for each $v,$ choose
\beq{whereTauvXiv}
(\tau_v, \xi_v) \in S_v \times \{0,1\}
\enq
with $\tau_v$ (which we take to be a first member of $L_v$) uniform from $S_v$
and, setting $\gzz = (1 - 1 / (D+1))$,
\beq{tauUpProb} 
\pr(\xi_v = 1 | \tau_v = \gc) = \gzz^{D-d_\gc(v)}.
\enq
Let
\beq{defT}
T = \{v : \tau_w \neq \tau_v \forall w \sim v ,\, \xi_v = 1\} ~\text{ and } ~\sigma = \tau|_T.
\enq
Thus we put $v$ in $T$ (and set $\gs_v=\tau_v$) if $\tau_w\neq \tau_v$ $\forall w\sim v$
\emph{and} $\xi_v=1$.  The $\xi$'s simulate the situation of \cite{APS}, in which
all $S_v $'s are $ [D+1] $:  writing $\tilde{\pr}$ for probabilities associated with that situation
we have 
\beq{inTeq}
\pr(v \in T) = \tilde{\pr}(v \in T)  \,\, (= \gzz^D) \ \ \forall v \in V.
\enq
(It may seem that skipping the $\xi_v$'s,
thus expanding the set $T$ of vertices handled by this initial phase, can only help;
but if $T\cap N_v$ is too large, 
then $L_v \setminus \tau(\{v\} \cup N_v)$, the list of still-legal colors for $v$,
will tend to be too small for use in Theorem~\ref{TRS}.)

We first show that \eqref{ideal1} and \eqref{ideal2} are correct at the 
level of expectations.
The expectation version of \eqref{ideal1},
\beq{expectTinNbd}
\E|T \cap N_v| \sim e^{-1} D,
\enq
follows from \eqref{inTeq}.
Our treatment of \eqref{ideal2} depends on relationships between the various lists $S_v$,
a departure from \cite{APS} where the lists are all the same.  Set
\beq{defFv}
F_v = \{ (uw, \gc) :uw\in \overline{G}[N_v] , \gc \in S_u \cap S_w \}
\enq
and
\beq{defAv}
A_v = \{(w, \gc): w\in N_v , \gc \in S_w \setminus S_v\}.
\enq
We first observe that \eqref{generic'} implies at least one of these is large:
\begin{prop}\label{sparseCases}
If $|\overline{G}[N_v]| \geq \vt D^2$ (for some $\vt$), then either
\beq{manyFraternals}
|F_v| \geq \vt D^3/2
\enq
or
\beq{manyAliens}
|A_v| \geq \vt D^2/2.
\enq
\end{prop}

\nin
(Of course in \cite{APS} there is only \eqref{manyFraternals}.)

\begin{proof}%[Proof of Proposition \ref{sparseCases}]
Write $\sum'$ 
for $\sum_{uw \in \overline{G}[N_v]}$ and set $d'(w) = |N_v\sm (N_w\cup \{w\})|.$
Supposing \eqref{manyAliens} fails, we have
(using, in order,
$|S_v| > D$,
\eqref{generic'},
$d'(w) < D$, 
and failure of \eqref{manyAliens} in the last four inequalities)
\begin{align*}
|F_v| &
= \sum\nolimits' |S_u \cap S_w| \geq \sum\nolimits' |S_v\cap S_u \cap S_w|
~\ge~ 
\sum\nolimits' [D - (|S_v \setminus S_w| + |S_v \setminus S_u|)] \\
&= \sum\nolimits' [D - (|S_w \setminus S_v| + |S_u \setminus S_v|)] 
~\ge ~\vt D^3 - \sum_{w \sim v} d'(w) |S_w \setminus S_v| 
~>~\vt D^3 - D |A_v| 
>~\vt D^3/2.
\end{align*}
\end{proof}

So we want to show that each of \eqref{manyFraternals}, \eqref{manyAliens}
implies an expectation version of \eqref{ideal2}.
Suppose first  that \eqref{manyFraternals} holds, and for $(uw, \gc) \in F_v$
define the event
\[
E_{uw, \gc}^v = \{u, w \in T, \tau_u = \tau_w = \gc \neq \tau_x \ \forall x \in N_v \sm \{u,w\}\}.
\]
For the l.h.s. of \eqref{ideal2} we have
\beq{fraternalLowerBound}
|T \cap N_v| - |\sigma(T \cap N_v) \cap S_v| 
\geq |\{(uw, \gc) \in F_v : E_{uw, \gc}^v \text{ holds}\}|
\enq
(since if $E_{uw, \gc}^v$ holds, then the pair $u,w$ contributes 2
to $|T \cap N_v|$ but at most 1 to $|\sigma(T \cap N_v) \cap S_v|$).  So 
we aim to establish \eqref{ideal2}
with its l.h.s.\ replaced by this lower bound; that is,
\beq{fraternalGetSlack}
|\{(uw, \gc) \in F_v : E_{uw, \gc}^v \text{ holds}\}|
> \vt' D \,\,\,\mbox{for all $v$ satisfying \eqref{manyFraternals};}
\enq
so we want an expectation version of this.

Let $N_\gc(u)=\{w\sim u:\gc\in S_w\}$, $N_\gc(U)=\cup_{u\in U}N_\gc(u)$ and 

\beq{defJuwgam}
J_{uw, \gc}^v = N_\gc(v,u,w) \setminus \{u,w\},
\enq
and note that
\beq{prEuw}
E_{uw, \gc}^v = \{\tau_u = \tau_w = \gc\} \land \{\tau_x \neq \gc \ \forall x \in J_{uw, \gc}^v\}
\land \{\xi_u = \xi_w = 1\} .
\enq
Thus
\beq{prEuw'}
\pr(E_{uw, \gc}^v) 
=(D+1)^{-2} \gzz^{|J_{uw, \gc}^v| + 2D - d_\gc(u) - d_\gc(w)}
\gtrsim e^{-3} D^{-2},
\enq
(where $|J_{uw, \gc}^v|$ and $2D - d_\gc(u) - d_\gc(w)$ correspond to the 2nd
and 3rd conjuncts in \eqref{prEuw}),
and combining with \eqref{manyFraternals} gives the desired
\beq{expectManyFraternalEvents}
\E|\{(uw, \gc) \in F_v : E_{uw, \gc}^v \text{ holds} \}| \gtrsim e^{-3} \vt D/2 =2\vt'D.
\enq

\mn

Now suppose \eqref{manyAliens} holds, and for $(w, \gc) \in A_v$,
define the event
\[
K_{w, \gc} = \{w \in T, \tau_w = \gc\}.
\]
For the l.h.s. of \eqref{ideal2} we now have
\beq{alienLowerBound}
|T \cap N_v| - |\sigma(T \cap N_v)\cap S_v| \geq |\{(w, \gc) \in A_v : K_{w, \gc} \text{ holds} \}|
\enq
(since if $K_{w, \gc}$ holds, then $w$ contributes 1 to $|T \cap N_v|$ and 
0 to $|\sigma(T \cap N_v) \cap S_v|$).
So we again aim for a strengthening of \eqref{ideal2}; viz.
\beq{alienGetSlack}
|\{(w, \gc) \in A_v : K_{w, \gc} \text{ holds} \}|
> \vt' D \,\,\,\mbox{for all $v$ satisfying \eqref{manyAliens}.}
\enq
Here for an expectation version we just note that
for $(w, \gc) \in A_v$,
\[
\pr(K_{w, \gc}) = (D+1)^{-1} \gzz^D \gtrsim e^{-1} D^{-1},
\]
so \eqref{manyAliens} implies
\beq{expectManyAlienEvents}
\E|\{(w, \gc) \in A_v : K_{w, \gc} \text{ holds} \}| \gtrsim e^{-1} \vt D/2 \,\,(=2 \vt'D).
\enq

\mn

So either of the possibilities in Proposition~\ref{sparseCases} supports an expectation version of 
\eqref{ideal2}; precisely,
\beq{Eideal2}
\E [|T \cap N_v| - |\sigma(T \cap N_v)\cap S_v| ] \gtrsim 2 \vt'D.
\enq

Given the above expectation versions,
we would
like to say, using martingale concentration and a union bound over $v$, that 
w.h.p.\ all left-hand sides in \eqref{ideal1} and \eqref{ideal2} 
are close to their expectations;  in which case, as suggested in Section~\ref{Sketch}
and formalized below, 
the lists $L_v\sm \{\tau_v\}$ will eventually support
use of Theorem~\ref{TRS}.  
This works for $n$ up to some $\exp[D^{1 - o(1)}]$ 
(recall
we may assume $(1+\gd)\log n< D+1$); but
for very large $n$ the concentration is not strong enough and 
we proceed less directly, as follows.

Say $n$ is \emph{large} (\emph{small}) if it satisfies (doesn't satisfy)
\beq{defLarge}
n > \exp[D^{0.9}].
\enq
In Section~\ref{getGoodTsigma} we will show
\beq{small23}
\mbox{\emph{if $n$ is small then w.h.p.\ the T and $\gs$ of \eqref{defT}
satisfy \eqref{ideal1} and \eqref{fraternalGetSlack} for all $v\in V^*$.}}
\enq

For larger $n$ we 
set $\ell = 0.1 \delta \log n$ and
choose the $\tau_v$'s in two steps:

\mn 
\textit{Step 1.} Each $v$ chooses (independently, uniformly) 
$L_v^0 \in {S_v \choose \ell}$, and we set $L^0 = (L_v^0 : v \in V)$. 

\nin
\textit{Step 2.} Each $v$ chooses (independently, uniformly)
$\tau_v\in L_v^0$ and $\xi_v$ as in \eqref{tauUpProb},
and we define $T, \sigma$ as in \eqref{defT}.

\mn
The crucial gain here is that we now need w.h.p.\ statements \emph{only for the} $L_v^0$'s:
it is enough to say that w.h.p.\ these partial lists support
\emph{existence}
of a good 
$T$ and $\gs$:  
\beq{CL0}
\mbox{\emph{if n is large, then w.h.p.\ 
$L^0$ is
such that w.p.p.\ the $T$ and $\gs$ of 
Step 2 satisfy \eqref{ideal1} and \eqref{ideal2} for all $v \in V^*$},}
\enq   

\nin
where \emph{w.p.p.---with positive probability}---just means the probability is not zero.
Thus the 2-step procedure succeeds because concentration for the relevant quantities output by
Step 1---those in Lemma~\ref{LPh1a}---\emph{is} strong enough to support union bounds, and the somewhat 
weaker concentrations of Step 2 are (far) more than enough for use with the 
Local Lemma.

For coloring
$V^*\sm T$, we then forget what we know about $L^0$
and work with fresh,
slightly smaller lists---thus the cost in list size mentioned preceding \eqref{ideal1}---so as to 
avoid the no longer comprehensible law of $L^0$.

In the rest of this section we set up the proof of \eqref{CL0}---namely, specifying, in 
Lemma~\ref{LPh1a}, the (w.h.p.) properties of $L^0$ that will, as asserted in Lemma~\ref{LC0},
support the ``w.p.p.'' statement of \eqref{CL0}---and 
then finish the sparse phase assuming these lemmas, whose proofs are postponed to
Section~\ref{Concentration}.
So for the next few paragraphs (through the statement of Lemma~\ref{LC0}) we 
continue to think of large $n$.

We now use $\pr'$ and $\E'$ for probabilities and expectations associated with Step 1,
and $\pr''$ for probabilities associated with Step 2; thus the
$\pr''(\cdots)$'s are r.v.s determined by $L^0$. We continue to write
$\pr$ and $\E$ for the probabilities and expectations associated
with $(\tau_v, \xi_v)$'s chosen as in \eqref{tauUpProb}.
The law of total expectation gives
\beq{LOTE}
\E' \pr''(\cdot) = \pr(\cdot) \  \text{ and } \ \E' \E''[\cdot] = \E[\cdot].
\enq

Let $\theta =o(1)$ and $C =\go(1)$ 
satisfy 
\beq{CtoThetaBig}
C^\theta D^{-\theta/C} = \omega(1),
\enq
plus the milder $C = D^{o(1)}$ 
and---\emph{only when n is large}---$C \theta^{-2} = o(\ell)$ and $\log(CD) = o(\theta \ell)$.
(E.g.\
$C=2\log D/\log\log D$ and any $\theta =\go(1/\log\log D)$
meet these requirements.)

\begin{lemma}\label{LPh1a}
W.h.p. (w.r.t. $\pr'$, i.e. 
the choice of $L^0$)
\beq{equivProbs}
\text{for all } v, \ \ \pr''(v \in T) \sim \pr(v \in T);
\enq
\beq{summablySimilarFraternals}
\text{for all } v, \ \ \left|\sum_{(uw, \gc) \in F_v} [\pr''(E_{uw, \gc}^v) - \pr(E_{uw, \gc}^v)] \right| = o(D);
\enq
\beq{summablySimilarAliens}
\text{for all } v, \ \ \left|\sum_{(w, \gc) \in A_v} [\pr''(K_{w,\gc}) - \pr(K_{w,\gc})] \right| = o(D);
\enq
\beq{L0smallFootprint}
\mbox{for all $v, Y \in {N_v \choose \theta D}$ and $ \Gamma_0 \sub \cup_{y\in Y}S_y$
with $|\gG_0| \leq \theta D/C$}, \ \ \
\sum_{y \in Y} |L_y^0 \cap \Gamma_0| < 2 \theta^2 D \ell / C;
\enq
and 
for all $v$ and $\pi: N_v \to \Gamma$ with $\pi_z\in S_z$ $\forall z$,
\beq{fewBystanders}
\sum_{w \not\sim v} |\pi(N_w \cap N_v) \cap L_w^0| < 2 D \ell.
\enq
\end{lemma}

\begin{lemma}\label{LC0}
If $L^0$ is as in Lemma \ref{LPh1a}, then w.p.p. Step 2 gives $T, \sigma$ satisfying 
\eqref{ideal1}, \eqref{fraternalGetSlack} and \eqref{alienGetSlack}.
\end{lemma}

The proofs of Lemmas~\ref{LPh1a} and \ref{LC0}
are given in Section~\ref{Concentration}, following a quick review of
what we need in the way of martingale concentration.
Once we have $T$ and $\gs$ as in \eqref{ideal1}-\eqref{ideal2}, 
we add colors (beyond the $\tau_v$'s for small $n$ or $L_v^0$'s for large)
and want to reduce the problem of extending $\gs$ to $V^*\sm T$ to an application of
Theorem~\ref{TRS}.  This is still a slightly long story, but is almost verbatim as in 
\cite[Sec.~4]{APS}, the only difference being that the new colors (called $L_v'$ in
\cite{APS}) are now chosen from the $S_v$'s rather than $[D+1]$; so we will not repeat it here.

\section{Martingales}\label{Martingales}

As in \cite{APS}, what we need from martingale concentration is
provided by \cite[Sec.\ 3]{aglc}. 
Our interest here is in concentration of a r.v.\ $X =X(\go)$, with
$\go=(\go_1\dots \go_m)$ drawn from some
$\gO =\prod_{i\in [m]}\gO_i$.
We will need two versions, the first of which assumes a product 
measure on $\gO$:

\begin{itemize}
\item $\omega_1 \in \Omega_1, \cdots,\omega_m \in \Omega_m$ are independent;
\item for each $i \in [m]$ and $\eps = (\eps_1, \cdots, \eps_m) \in \Omega$
we are given some $\Sigma_i = \Sigma_i(\eps_1, \cdots, \eps_{i-1}) \subseteq \Omega_i$.
\end{itemize}

\nin
(In \cite{APS} we assume 
$\go_i$ is uniform from $\gO_i$, but the proof of \cite[Lemma~5.1]{APS},
which apart from this tiny change is the present Lemma~\ref{martConc},
makes no use of this assumption.)
The second version, for use when we come to choosing $L_w^0$'s
(in Section~\ref{L0HighProbProof}), assumes:

\begin{itemize}
\item $[m] = U \times [\ell]$ ($U$ some set), ordered lexicographically;
\item for $i = (z,k)$, $\Omega_i$ depends only on $z$ (so we write $\Omega_z$);
\item each $\omega_{z,k}$ is uniform from $\Omega_z \setminus \{\omega_{z,1}, \cdots, \omega_{z,k-1}\}$,
and is independent of all $\omega_{x,j}$ with
$x \neq z$;
\item $X(\omega)$ depends only on the \emph{sets} $(\{\omega_{z,1}, \cdots, \omega_{z,\ell}\} : z \in U)$;
\item for each $z \in U$ we are given $\Sigma_z \subseteq \Omega_z$ and for $i = (z,k)$ set 
$\Sigma_i = \Sigma_z$.
\end{itemize}

\nin
We will refer to these two situations as (A) and (B) respectively. The following lemma applies to both.
\begin{lemma}\label{martConc}
If, for all $i \in [m]$, $\eps_i^0, \eps_i^1 \in \Omega_i$,
and $\eps_j \in \Omega_j$ (for $j \neq i$),
\beq{lipschitzPlus}
|X(\eps_1, \cdots, \eps_{i-1}, \eps_i^0, \eps_{i+1}, \cdots, \eps_m) - X(\eps_1, \cdots, \eps_{i-1}, \eps_i^1, \eps_{i+1}, \cdots, \eps_m)|
\begin{cases}
\leq c, \\
= 0 \ \ \ \text{ if } \eps_i^0, \eps_i^1 \notin \Sigma_i,
\end{cases}
\enq
and
\beq{unlikelyActive}
\sum_i \pr(\omega_i \in \Sigma_i | \omega_j = \eps_j \text{ for } j < i) \leq M,  
\enq
then for any $\lambda \in [0,16cM]$,
\beq{martBound}
\max\{\pr(X-\E X > \lambda), \pr(X - \E X < -\lambda) \} < \exp \left[ -\frac{\lambda^2}{32 c^2 M} \right].
\enq
\end{lemma}
\nin
(The restriction on $\gl$ will never be a problem, and 
will usually be left to the reader.)
See \cite[Sec.\ 5]{APS} for discussion and a quick 
guide to extracting Lemma~\ref{martConc} from
\cite[Sec.\ 3]{aglc}.

\section{Concentration}\label{Concentration}
The eventual goal here, addressed in Section \ref{getGoodTsigma},
is establishing \eqref{ideal1} and \eqref{ideal2};
but we first, in Section \ref{L0HighProbProof}, dispose of Lemma \ref{LPh1a},
which underlies the analysis for large $n$.

In Section \ref{L0HighProbProof} 
and the part of Section \ref{getGoodTsigma} concerning small $n$,
we aim for union bounds over sets of vertices of bounded size,
so want to say that each of various desired events fails
with probability $n^{-\omega(1)}$; in these cases we use ``w.v.h.p."
to mean with probability $1 - n^{-\omega(1)}$.
In this section, ``proof of \eqref{equivProbs}" (e.g.) is
short for ``proof that \eqref{equivProbs} holds with suitable probability".

As often with martingale concentration, formal justifications can take a while
even when intuition is clear, so a reader might profitably skip most of this section;
e.g.\ cursory examinations of the 
proofs of \eqref{summablySimilarFraternals}
(beginning on the line before \eqref{Yvuw}) and
\eqref{ideal1} for large $n$
(beginning a couple paragraphs before \eqref{TNv'}) should be more than enough
to see the ideas here.

\subsection{Proof of Lemma \ref{LPh1a}}\label{L0HighProbProof}
This is all based on regime (B) of Lemma~\ref{martConc}.
Here we will always have
\[
U \subseteq V, \,\,
\Omega_z = S_z
\]
and 
\beq{goz}
\{\omega_{z,1}, \cdots, \omega_{z,\ell}\} = L_z^0 ;
\enq
and may use the lemma with any convenient
\beq{mgConcBM}
M \ge ~ \ell \sum_{z \in S} |\Sigma_z| / (D - \ell),
\enq
(the summand in (\ref{unlikelyActive})
for $i = (z,k)$ being at most $|\Sigma_z|/(D-k+1)$).

We use $N_\gamma^0(v) = \{w \sim v : \gamma \in L_w^0\}$,
with $N_\gamma^0(W)$ the analogue of $N(W)$ and $d_\gamma^0(\cdot) = |N_\gamma^0(\cdot)|$.
This section will see frequent appearances of $1-1/(D+1)$, which we continue to abbreviate:
\beq{gzz}
\gzz := 1-1/(D+1).
\enq
\begin{obs}\label{binomPowerExpect}
Fix $\gamma \in \Gamma$. Let
$W \subseteq \{v \in V : \gamma \in S_v\}$
and $W^0 = \{v \in W : \gamma \in L_v^0\}.$ Then
\[
\E'\left[(1 - \ell^{-1})^{|W^0|}\right] = \gzz^{|W|}.
\]
\end{obs}
\begin{proof}
This is just 
$~\E' \pr''(\gc\not\in\tau(W)) = \pr(\gc\not\in\tau(W)),$
an instance of \eqref{LOTE}.
\end{proof}

\nin \textit{Proof of \eqref{equivProbs}.}
Since $\E'[\pr''(v \in T)] = \pr(v \in T) \sim e^{-1}D$ (see \eqref{LOTE} and \eqref{inTeq})
we just need concentration of
\beq{concP''}
\pr''(v\in T) = \sum_{\gamma \in L_v^0} \left[ \ell^{-1} (1 - \ell^{-1})^{d_\gamma^0(v)} 
\gzz^{D - d_\gamma(v)} \right].
\enq
We first show the expected value of this probability doesn't depend on $L_v^0$:

\nin \textit{Claim:} For each $\Sigma \in {S_v \choose \ell}$,
\beq{conditionalEqual}
\E'[\pr''(v \in T) | L_v^0 = \Sigma] = \E'[\pr''(v \in T)] \ \ (= \pr(v \in T) =\gzz^D).
\enq
(In \cite{APS} this just follows from symmetry.)

\nin \textit{Proof.}
Recalling \eqref{concP''} and using
linearity of expectation, we find that the l.h.s.\ of \eqref{conditionalEqual} is
\beq{lofE}
%\E'[\pr''(v \in T) | L_v^0 = \Sigma] = 
\sum_{\gamma \in \Sigma} \ell^{-1} 
\gzz^{D - d_\gamma(v)}\cdot
\E'[(1 - \ell^{-1})^{d_\gamma^0(v)}].
\enq
(Note $d_\gamma^0(v)$ doesn't depend on $\gS$.)
On the other hand, Observation \ref{binomPowerExpect}
(used with $W = N_\gamma(v)$) says
\[
\E'[(1 - \ell^{-1})^{d_\gamma^0(v)}] = \gzz^{d_\gamma(v)},
\]
and inserting this in \eqref{lofE} gives, as desired,
$\E'[\pr''(v \in T) | L_v^0 = \Sigma] =\gzz^D$.
\qed

\mn

So for \eqref{equivProbs} it is enough to show concentration of $\pr''(v \in T)$ given 
$L_v^0=\gS$ (for any $\gS$).  
Given $L_v^0=\gS$, the sum in \eqref{concP''} is a function of the events
$\{\gc\in L_y^0\}$,
with 
$\gc$ running over $\gS$ and $y$ over $N_z$.
We use Lemma~\ref{martConc} (regime (B)) with $X=\pr''(v \in T)$, $\go$ as in \eqref{goz},
and
\beq{SSgS}
\mbox{$U = N_v, \  \ \Omega_z = S_z, \ \ \Sigma_z = \Sigma \cap S_z, \ \ c = \ell^{-2}\,$ 
and $\, M = 2\ell^2$.}
\enq
[Here \eqref{lipschitzPlus}
follows from \eqref{concP''}:
changing $\go_i=\go_{z,k} $ from $\gc$ to $\gd$ has no effect if $\gc,\gd\not\in \gS$ 
and otherwise may increase and/or decrease a single summand in \eqref{concP''}
by at most $\ell^{-1}[1-(1-\ell^{-1})]$.
For \eqref{unlikelyActive} note that 
the r.h.s.\ of
\eqref{mgConcBM} is at most
$\ell \cdot D \cdot 2\ell/D$, i.e., is bounded by the present $M$.]
Then for any $\lambda = \omega(\ell^{-1/2})$ the lemma gives
\[   
\pr'(|\pr''(v \in T) - \pr(v \in T)| > \lambda | L_v^0 = \Sigma)
< e^{-\omega(\ell)} = n^{-\omega(1)},
\]
so taking $\lambda$ to be $\omega(\ell^{-1/2})$ 
and $o(1)$ gives
\eqref{equivProbs}.

\qed

\mn 
\textit{Proof of \eqref{summablySimilarFraternals}.}
Set $S_{uw} = S_u \cap S_w$
and $L_{uw}^0 = L_u^0 \cap L_w^0.$
For $uw \in \overline{G}[N_v]$ let
\beq{Yvuw}
Y^v_{uw} = \sum_{\gamma \in S_{uw}} \pr''(E_{uw,\gamma}^v),
\enq
\beq{deffvuw}
f_{uw}^v = f_{uw}^v(L_{uw}^0) 
= \E'[Y^v_{uw}  | L_{uw}^0]
\enq
(note that once we know $L_{uw}^0$ the sum in \eqref{Yvuw} 
might as well be over $\gc \in L_{uw}^0$), 
and
\beq{Xv}
X_v = \sum_{uw \in \overline{G}[N_v]} f_{uw}^v.
\enq
By \eqref{LOTE} and linearity
of expectation,
\beq{fullSumItExp}
\sum_{\gamma \in S_{uw}} \pr(E_{uw,\gamma}^v)
~= ~\E'[Y^v_{uw} ] ~= ~ \E' \{\E'[Y^v_{uw} | L_{uw}^0 ]\} 
~= ~\E'[f_{uw}^v].
\enq
So the l.h.s.\ of \eqref{summablySimilarFraternals} is
\begin{align}
\nonumber \left|
\sum_{(uw,\gamma) \in F_v} 
\left[
\pr''(E_{uw,\gamma}^v) - \pr(E_{uw,\gamma}^v)
\right]
\right| &= \left|
\sum_{uw \in \overline{G}[N_v]}
(Y^v_{uw}- \E'[f_{uw}^v])
\right| \\
\label{fraternalInnOut} &\leq
\sum_{uw \in \overline{G}[N_v]}
|
Y^v_{uw}- f_{uw}^v
|
+ \left|
X_v - \E' X_v
\right|,
\end{align}
and we just need to show
\beq{fraternalInnOutSmall}
\text{each of the two parts of \eqref{fraternalInnOut}
is } o(D) \text{ w.v.h.p.}
\enq

For the first part, we can prove a stronger ``local'' version:  it is enough to show,
for some $\gl=o(1/D)$ and any (relevant) $uw$ and $\gS$,
\beq{Evuwtoshow}
\pr'(|Y^v_{uw}- f_{uw}^v(\Sigma)| > \lambda
\ | \ L_{uw}^0 = \Sigma ) =n^{-\go(1)}.
\enq

Set 
\[    %\beq{defJuwgam0}
J_{uw,\gamma}^{v,0} = N_\gamma^0(u,w,v) \setminus \{u,w\}
\,\,\,\,(= \{z\in J_{uw,\gamma}^v: \gc\in L^0_z\}; \,\text{see \eqref{defJuwgam}}).
\]   %\enq
Fix $uw \in \overline{G}[N_v]$.  With $Y^v_{uw}$ as in \eqref{Yvuw}, 
we have (cf.\ \eqref{prEuw'})
\beq{uwPpp}
Y^v_{uw} \, 
= \sum_{\gamma \in L_{uw}^0} \ell^{-2} 
(1 - \ell^{-1})^{|J_{uw,\gamma}^{v,0}|}\cdot
\gzz^{2D - d_\gamma(u) - d_\gamma(w)}.
\enq
Conditioning on $\{L_{uw}^0 = \Sigma\}$, we apply
Lemma~\ref{martConc} (B) with $X=Y^v_{uw}$
and
\beq{mgParamsFraternalOuter}
U = N(u,w,v) \setminus \{u,w\}, \ 
\Omega_z = S_z, \ 
\Sigma_z = \Sigma \cap S_z, \ 
c = \ell^{-3} \ \text{ and }
M = 4\ell^2.
\enq
[Justification of 
\eqref{lipschitzPlus} and \eqref{mgConcBM} for these parameters
is similar to the discussion following \eqref{SSgS}.]
Inserting these in the lemma 
(and recalling from \eqref{deffvuw} that
$f_{uw}^v(\gS) = \E'[X| L_{uw}^0=\gS]$),
we have for the l.h.s.\ of \eqref{Evuwtoshow}
(and $\lambda \leq 64/\ell$)
\[
\pr'(|Y^v_{uw}- f_{uw}^v(\Sigma)| > \lambda
\ | \ L_{uw}^0 = \Sigma )  <
\exp[-\Omega(\lambda^2 \ell^4)].
\]
Taking $\gl$ to be 
both $\omega(\ell^{-3/2})$ and $o(1/D)$
(as we may since we assume \eqref{defLarge})
thus gives \eqref{Evuwtoshow}.

\mn

We turn to the second part
of \eqref{fraternalInnOut}.  
Let
\[
F_v^0 = \left\{(uw,\gamma) \in \overline{G}[N_v]
\times \Gamma : \gamma \in L_{uw}^0 \right\}.
\]
Recalling \eqref{deffvuw} and the expression for $Y^v_{uw}$ in \eqref{uwPpp}, 
and using Observation \ref{binomPowerExpect}
with $W = J_{uw,\gamma}^v$,
we have
\beq{binomAppTofProbs}
X_v ~= \sum_{uw \in \overline{G}[N_v]} f_{uw}^v 
~= \sum_{uw \in \overline{G}[N_v]}
\E' [ Y^v_{uw}~|~
L_{uw}^0 ] 
~= \sum_{(uw,\gamma) \in F_v^0}
\ell^{-2} \cdot\gzz^{|J_{uw,\gamma}^v| + 2D-d_\gamma(u) - d_\gamma(w)}.
\enq
We apply Lemma~\ref{martConc} (B)
with $X = X_v$ and
\[
U = N_v, \ \Omega_z = \Sigma_z = S_z,\ 
c = D \ell^{-2} \ \text{ and } \ M = D\ell.
\]
[Justification for \eqref{lipschitzPlus} is again similar to that following \eqref{SSgS};
here changing $\go_{z,k} $ can add and/or delete at most $D-1$ terms 
from the sum on the r.h.s.\ of \eqref{mgConcBM}, and the summands are at most $\ell^{-2}$.
And $D\ell$ bounds the r.h.s.\ of \eqref{mgConcBM} 
since each of the $D$ terms in the sum there is at most 1.]
This gives (for $\lambda \leq 16D^2 \ell^{-1}$)
\[
\pr'(|X_v - \E' X_v| > \lambda) < \exp\left[-\Omega(\lambda^2 (\ell/D)^3)\right],
\]
and taking $\lambda$ both 
$\omega(D^{3/2} \ell^{-1})$ and $ o(D)$
(again allowed by \eqref{defLarge})
gives the second part of 
\eqref{fraternalInnOutSmall}.

\qed

\nin \textit{Proof of \eqref{summablySimilarAliens}.}
Set
\beq{defKw}
K_w = \bigcup_{\gamma \in S_w \setminus S_v} K_{w, \gamma}
\,\,\,\,(=\{w\in T, \tau_w\not\in S_v\});
\enq
so
\[
\pr(K_w) = \sum_{\gamma \in S_w \setminus S_v} \pr(K_{w,\gamma})
\]
and similarly with $\pr$ replaced by $\pr''$.
Let
\beq{defew}
k_w = k_w(L_w^0) = \E' \left[ \pr''(K_w) | L_w^0 \right]
\enq
(for $w \in N_v$) and
\[
Y_v = \sum_{w \sim v} k_w.
\]
Since (again by \eqref{LOTE})
\[
\pr(K_w) = \E'\left\{ \E' [\pr''(K_w) | L_w^0] \right\} = \E'[k_w],
\]
we may rewrite the l.h.s.\ of \eqref{summablySimilarAliens} as
\begin{align}
\nonumber 
\left|
\sum_{w \sim v} \left(\pr''(K_w) - \E'k_w\right)
\right| 
&= \left|
\sum_{w \sim v} (\pr''(K_w) - k_w + k_w - \E' k_w)
\right| \\
\label{alienInnOut} &\leq \sum_{w \sim v} |\pr''(K_w) - k_w| + |Y_v - \E' Y_v|,
\end{align}
so just need to show
\beq{alienInnOutSmall}
\text{each of the two parts of \eqref{alienInnOut} is $o(D)$ \text{ w.v.h.p.}}
\enq

For the first part, we again aim for a ``local'' version:  
for some $\gl=o(1)$ and $\gSR\sub S_w $ of size $\ell$,
\beq{Evuwtoshow'}
\pr'(|\pr''(K_w) - k_w(R)| > \lambda | L_w^0 = \gSR)=n^{-\go(1)}.
\enq
Noting that
\beq{probppKw}
\pr''(K_w) = \sum_{\gamma \in S_w \setminus S_v} \ell^{-1} (1 - \ell^{-1})^{d_\gamma^0(w)} 
\cdot \gzz^{D - d_\gamma(w)},
\enq
we condition on $\{L_w^0 = \gSR\}$ and
apply Lemma~\ref{martConc}
with $X = \pr''(K_w)$ and
\[
U = N_w,\  \Omega_z = S_z,\  \Sigma_z = R \cap (S_z \setminus S_v),\  c = \ell^{-2},\  
\text{ and } M = 2\ell^2.
\]
[Here \eqref{lipschitzPlus} follows from \eqref{probppKw} and the summands in
\eqref{mgConcBM} are trivially less than $2\ell/D$.]
This gives (now for \emph{any} $\lambda \geq 0$, since for $\gl>1$ the l.h.s.\ is zero)
\[
\pr'(|\pr''(K_w) - k_w(R)| > \lambda | L_w^0 = R) < \exp[-\Omega(\lambda^2 \ell^2)],
\]
and taking $\lambda$ to be both $\omega(\ell^{-1/2})$ and $o(1)$
gives \eqref{Evuwtoshow'}.

\mn

For the second part of \eqref{alienInnOut}, 
recall that $A_v$ was defined in \eqref{defAv} and set 
$A_v^0 = \{(w,\gamma) \in A_v : \gamma \in L_w^0\}$.
Applying Observation~\ref{binomPowerExpect} to the $N_{\gamma}(w)$'s gives
\[
Y_v = |A_v^0| \ell^{-1} \gzz^D,
\]
a function of $L^0|_{N_v}$.
(Note that $d_\gc^0(w)$ is independent of membership 
of $(w,\gc)\in A_v^0$.)  
Applying Lemma~\ref{martConc}~(B) (excessive here, where
$|A_v^0| =\sum_{w\sim v}|L_w^0\cap (S_w\sm S_v)|$ is a sum of independent
hypergeometrics)
with
\[
U = N_v,\  \Omega_z = \Sigma_z = S_z,\  c = \ell^{-1},\  \text{ and } M = 2D\ell
\] 
gives
\[
\pr'(|Y_v - \E' Y_v| > \lambda) < \exp[-\Omega(\lambda^2 \ell/D)];
\]
so taking $\lambda$ to be both
$\omega(D^{1/2}) $ and $o(D)$ gives
the second part of \eqref{alienInnOutSmall}. 

\qed

\mn

\nin \textit{Proof of \eqref{L0smallFootprint}.}
Fix $v, Y, \Gamma_0$ as in (\ref{L0smallFootprint}).
For the quantity of interest, $X = \sum_{y \in Y} |L_y^0 \cap \Gamma_0|$,
we have
\[
\E'X \leq \theta D \cdot \theta D/C \cdot \ell/D = \theta^2 D \ell/C,
\]
so again just need concentration.
Applying Lemma~\ref{martConc} (B) with
\[
U = Y,\  \Omega_y = S_y,\  \Sigma_y = \Gamma_0 \cap S_y,\  c = 1,\  \text{ and } M = |Y|\ell \cdot |\Gamma_0|/(D-\ell) < 2 \theta^2 D \ell/C,
\]
gives
\beq{unlikelyBigFootprint}
\pr'(X > \E'X + \theta^2 D \ell / C) < \exp[-\Omega(\theta^2 D \ell/C)] = e^{-\omega(D)},
\enq
where the r.h.s.\ uses
$C \theta^{-2} = o(\ell)$
(see following \eqref{CtoThetaBig}).
On the other hand, the number of possibilities for $(v,Y,\Gamma_0)$ is at most
\[
n \C{D}{\theta D} \C{\theta D (D+1)}{\leq \theta D/C} ,
\]
which with \eqref{unlikelyBigFootprint} bounds the probability that \eqref{L0smallFootprint}
does not hold by 
\[
n 2^D O(CD)^{\theta D/C} e^{-\Omega(\theta^2 D \ell/C)} = e^{-\omega(D)} = n^{-\omega(1)},
\] 
where the first ``='' uses 
$\log(CD) = o(\theta \ell)$;  again, see following \eqref{CtoThetaBig}.

\qed

\nin \textit{Proof of \eqref{fewBystanders}.}
Fix $(v,\pi)$ and let $X = \sum_{w \not\sim v} |\pi(N_w \cap N_v) \cap L_w^0|$
(the sum in \eqref{fewBystanders}),
noting that
\[
\E' X \leq\ \sum_{w \not\sim v} |\pi(N_w \cap N_v)| \ell/(D+1)
< \sum_{w \not\sim v} |N_w \cap N_v| \ell/D \,  \leq \, D\ell.
\]

\nin
(Of course $w$'s at distance more than $2$ from $v$
play no role here.) For Lemma~\ref{martConc} (B) we take
\[
U = \{w : \text{dist}(v,w) \in \{0,2\}\},\  \Omega_z = S_z,\  \Sigma_z = \pi(N_z \cap N_v) \cap S_z,\  c = 1, 
\]
and
\[
M ~=~ 2 D \ell ~\geq~ \ell \sum_{w \not\sim v} |N_w \cap N_v|/(D -\ell)
~\geq ~\ell \sum_{w \not\sim v} |\pi(N_w \cap N_v)|/(D -\ell),
\]
an upper bound on the r.h.s.\ of \eqref{mgConcBM}.
This gives
\[
\pr'(X > 2D\ell) < \exp[-\Omega(D \ell)],
\]
and \eqref{fewBystanders} follows, since there are only $n(D+1)^D$ possibilities for $(v,\pi)$
(and $n$ is large). \qed

\subsection{Properties \eqref{ideal1} and \eqref{ideal2}}\label{getGoodTsigma}
This section proves \eqref{small23} and Lemma~\ref{LC0}.
Here we will find it convenient to take $V = [n]$ and \emph{sometimes}
use $i,j,\cdots$ instead of $v,w,\cdots$ for vertices.

We treat small and large $n$ in parallel.
When dealing with large $n$, we are given $L_v^0$'s, which we may
(and do) assume satisfy \eqref{equivProbs}-\eqref{fewBystanders}.

For $\pi: N_v \to \Gamma$ with $\pi_w \in S_w$ for all $w$
($\pi$ will later be $\tau |_{N_v}$),
let
\beq{defZvtau}
Z(v,\pi)  = \{z \in N_v : |\pi^{-1}(\pi_z) \cap N_v| < C\},
\enq
and say $(v,\pi)$ is \emph{bad} if $|Z(v,\pi)| < (1-\theta)D$
(and \emph{good} otherwise).

Before attacking \eqref{ideal1} and \eqref{ideal2},
we show that bad $(v,\tau |_{N_v})$'s 
(for $\tau$ see \eqref{whereTauvXiv} for small $n$
and \textit{Step 2} above \eqref{CL0} for large)
are unlikely; this will be used to control the differences
in \eqref{lipschitzPlus} of Lemma~\ref{martConc}.
We show, for both small and large $n$ and any $v$,
\beq{vtaubad}
\pr(\mbox{$(v,\tau|_{N_v})$ is bad}) =e^{-\go(D)}.
\enq

\begin{proof}
Fix $v$ and notice that if $(v,\tau|_{N_v})$ is bad
then there are $Y \in {N_v \choose \theta D}$
and 
\[
\Gamma_0 \in {\cup_{y \in Y} S_y \choose \theta D/C}
\]
with 
$\tau(Y) \subseteq \Gamma_0$; thus for small $n$,
\begin{align}
\label{vTNvUnlikelyBadsmalln}
\pr((v,\tau|_{N_v}) \text{ bad}) &< {D \choose \theta D} %choose Y
{\theta D (D+1) \choose \theta D/C} %choose \Gamma_0 from \cup_{y \in Y} S_y
\left(\frac{\theta D}{(D+1)C} \right)^{\theta D} \\
\nonumber &< \left[(e/\theta) (e C D)^{1/C} (\theta/C) \right]^{\theta D} \\
\nonumber &= [O(C^{-1} D^{1/C})]^{\theta D} = e^{-\omega(D)}
\end{align}
(using \eqref{CtoThetaBig}),
while for
large $n$, \eqref{L0smallFootprint}  allows us to replace the last term on the r.h.s.\ of \eqref{vTNvUnlikelyBadsmalln} by
\[
\prod_{y \in Y} \frac{|L_y^0 \cap \Gamma_0|}{\ell} \leq 
\left(\frac{1}{\theta D} \sum_{y \in Y} \frac{|L_y^0 \cap \Gamma_0|}{\ell} 
\right)^{\theta D}
\leq \left(\frac{2\theta}{C}\right)^{\theta D}.
\]

%%%%%%%%%
\iffalse
\begin{align}
\label{vTNvUnlikelyBadLargen}
\pr''((v,\tau|_{N_v}) \text{ bad}) &< \sum_{Y} \sum_{\Gamma_0} 
\prod_{y \in Y} \frac{|L_y^0 \cap \Gamma_0|}{\ell} \\
\nonumber &\leq {D \choose \theta D} {\theta D(D+1) \choose \theta D/C} 
\left(\frac{1}{\theta D} \sum_{y \in Y} \frac{|L_y^0 \cap \Gamma_0|}{\ell} 
\right)^{\theta D} \\
\nonumber &= [O(C^{-1} D^{1/C})]^{\theta D} = e^{-\omega(D)}.
\end{align}
\fi
%%%%%%%
\end{proof}

For the rest of this section we fix
$v \in V^*$ (as mentioned earlier,
this restriction isn't needed
for \eqref{ideal1}); let
$\tau = (\tau_1, \cdots, \tau_n)$, with the $\tau_i$'s
independent and
\[
\tau_i \text{ uniform from } \begin{cases}
S_i \ \ \ \text{ if } n \text{ is small,} \\
L_i^0 \ \ \text{ if } n \text{ is large;}
\end{cases}
\]
and let $\xi = (\xi_1, \cdots, \xi_n)$
with the $\xi_i$'s as in \eqref{tauUpProb}.
For convenience we assume $N_v = [D]$
and use whichever of $N_v$, $[D]$
seems more appropriate, always using
$\tau_{[D]}$ and $\xi_{[D]}$ for the restrictions
of $\tau$ and $\xi$ to $N_v$.

The arguments in the rest of this section 
are applications of regime (A) of Lemma~\ref{martConc}.  
For these we will always take
$\Omega = \prod_{i \in [n]} \Omega_i$,
noting that most $i$'s---those
with dist$(v,i) > 2$---will never affect
the $X$ in question; such $i$'s will have
$\Sigma_i = \emptyset$, so won't affect $M$
(see \eqref{unlikelyActive}) or the bound \eqref{martBound}.
(We could instead take $\Omega = \prod_{i \in [m]} \Omega_i$,
with $[m] = \{i : \text{dist}(v,i) \leq 2\}$.)

\nin \textit{Proof of \eqref{ideal1} for small $n$.}
Let 
\[
X = X(\tau, \xi) = |T \cap Z| = \sum_{i \in N_v} X_i,
\]
where $Z = Z(v,\tau_{[D]})$ (see \eqref{defZvtau}) and $X_i = \mathbbm{1}_{\{i \in T \cap Z \}}$.

We first observe that
\beq{correctExpect}
\E X \sim e^{-1} D
\enq
and that we will have \eqref{ideal1}, precisely,
\beq{TNv}
\text{w.v.h.p. } |T \cap N_v| \sim e^{-1} D, 
\enq
if we can show
\beq{goodColorConcsmalln}
\text{w.v.h.p. } |X - \E X| = o(D).
\enq
[Because: with $U = N_v \setminus Z$,
%the combination of \eqref{vTNvUnlikelyBadsmalln}
%and \eqref{CtoThetaBig} 
\eqref{vtaubad}
and $\theta =o(1)$ (and $\E|U| < \theta D + D ~\cdot \pr(\mbox{$(v,\tau|_{N_v})$ is bad})$) 
give
\beq{fewLeftovers}
\E|U| \lesssim \theta D \ \text{ and w.v.h.p. } |U|  \le \theta D,
\enq
which, since $0 \leq|T \cap N_v| - X \leq |U|$,
implies (\ref{correctExpect}) (see (\ref{expectTinNbd})) and
\beq{closeToGoodCount}
\text{w.v.h.p. } ||T \cap N_v| - X|  \le \theta D.]
\enq

\nin \textit{Proof of \eqref{goodColorConcsmalln}.}
Notice (with justification to follow) that if
$\tau_i' = \beta$ and $(\tau_j,\xi_j) = (\tau_j',\xi_j') \ \forall j \neq i$,
then
\beq{lipschitzForWellColoreds}
X(\tau, \xi) - X(\tau', \xi') \leq \begin{cases}
C \ \text{ (always),} \\
0 \ \text{ if } i > D \text{ and } i \not\sim \tau^{-1}(\beta) \cap Z.
\end{cases}
\enq
(Note $Z$ is determined by $\tau_{[D]}$, so its appearance here makes sense.)
This is true because, with $X = X(\tau,\xi)$,
$X' = X(\tau',\xi')$, and $j \in [D]$, we can only have $X_j > X_j'$
(i.e. $X_j' = 0, X_j = 1$) if
\begin{itemize}
\item $i=j$, or
\item $i \sim j$, $\tau_j = \beta$, and $|\tau^{-1}(\beta) \cap (N_v \setminus \{i\})| < C$, or
\item $i \leq D$, $\tau_j = \beta$, and $|\tau^{-1}(\beta) \cap (N_v \setminus \{i\})| = C-1$.
\end{itemize}
For Lemma~\ref{martConc} (regime (A)) we take
\begin{itemize}
\item $\Omega_i = S_i \times \{0,1\}$, $\omega_i = (\tau_i, \xi_i)$;
\item for $i \in [D]$, $\Sigma_i = \Omega_i$ 
(so the summand in (\ref{unlikelyActive}) is $1$);
\item for $i > D, \Sigma_i = \Sigma_i(\tau_{[D]}) = (\tau(N_i \cap N_v) \cap S_i) \times \{0,1\}$
(here $N_v$ could be $\{x \in N_v : \xi_x = 1\}$); and
\item $c = C$.
\end{itemize}
Then (\ref{lipschitzPlus}) is given by (\ref{lipschitzForWellColoreds}),
and for $i>D$,
\beq{activeBound}
\pr(\omega_i \in \Sigma_i) = |\tau(N_i \cap Z) \cap S_i|/(D+1) \leq |N_i \cap Z|/D,
\enq
implying that the sum in (\ref{unlikelyActive}) is at most
\[
D + D^{-1} \sum_{i > D} |N_i \cap N_v| \leq 2D =: M.
\]
Thus, finally, Lemma~\ref{martConc} gives
\[
\pr(|X - \E X| > \lambda) = \exp[-\Omega(\lambda^2/(C^2 D))];
\]
so taking $\gl$ to be both $\omega(C (D \log n)^{1/2})$ and $o(D)$
(possible since $n$ is small and, as promised following
\eqref{CtoThetaBig}, $C=D^{o(1)}$), we have \eqref{goodColorConcsmalln}.

\qed

\nin \textit{Proof of \eqref{ideal1} for large $n$.}
Recall that the $\tau_i$'s are now drawn uniformly
from the $L_i^0$'s, which are assumed to satisfy
(\ref{equivProbs})-(\ref{fewBystanders})
the $\xi_i$'s are still as in (\ref{tauUpProb}),
and 
we aim for Lemma~\ref{LC0}; that is, we just
need to say, using the Local Lemma (Lemma \ref{LLL}), that with
\emph{positive} probability (\ref{ideal1})
holds for all $v \in V^*$. (So, we are applying the lemma
to events $A_v = \{|T \cap N_v| \neq (1\pm \eta) e^{-1} D\}$, with $\eta>0$ an
arbitrarily small constant or a slow enough $o(1)$.
Again, the restriction to
$V^*$ is unnecessary here.)

\nin

%%%%%%%%%%
\iffalse
\cha{[An attempt to spell out how we're applying LLL]
The application of Lemma \ref{LLL} works as follows.
We take $m = n$, and the bad events $(A_i : i \in [m])$ are indexed
by vertices of $V$.
With an implicit $\iota = o(1)$,
for $i \in V$ we take
$$A_i = \{|T \cap N_v| \neq (1\pm \iota) e^{-1} D\}.$$
Put $i \sim_{\frak{G}} j$ if $\text{dist}_G(i,j) \leq 4$:
then, $A_i$ is independent (under $\pr''$) of all $A_j$ for which 
$i \not\sim_{\frak{G}} j$, since the underlying $(\tau_v, \xi_v)_{v \in V}$ are independent
and $A_i$ is determined by $(\tau_v, \xi_v)_{\text{dist}_G(v,i) \leq 2}$. Thus
$$\Delta_{\frak{G}} \leq \max_{v \in V} |\{w \in V : \text{dist}_G(v,w) \leq 4\}| \leq D^4.$$
We will show
$$\pr''(A_i) \leq p := D^{-5}\,\, \forall i \in V,$$
giving $e p (\Delta_{\frak{G}} +1) <1$ with plenty of room.
Thus the hypothesis of Lemma \ref{LLL} is satisfied, and we get
$$\pr''(\cap \bar{A_i}) = \pr''(\forall v \in V,\,\, |T \cap N_v| = (1\pm \iota)e^{-1} D) > 0.$$
We now use \emph{w.f.h.p.} (f = fairly) for ``with probability
at least $1 - D^{-5}$."}
\fi
%%%%%%%%%%%%%%%

Thus, since each event in (\ref{ideal1}) is independent
of all but at most $D^4$ others (since each is determined
by the restriction of $(\tau,\xi)$ 
to vertices within distance $2$
of $v$), even failure probabilities below $D^{-5}$
suffice here (though we will do much better),
so we now use \emph{w.f.h.p.} (f = fairly) for 
``with probability at least $1 - D^{-5}$."
The arguments here are similar to those for small $n$,
so we mainly indicate differences. 

As before,
we have $v \in V = [n]$, $N_v = [D]$,
and $\tau = (\tau_1, \cdots, \tau_n)$, and set
\[
X = X(\tau, \xi) = |T \cap Z| = \sum_{z \in N_v} X_z, 
\]
where $Z = (v,\tau_{[D]})$ and $X_z = \mathbbm{1}_{\{z \in T \cap Z\}}$.
The desired counterpart of \eqref{TNv} is 
\beq{TNv'}
\text{w.f.h.p. } |T \cap N_v| \sim e^{-1} D, 
\enq
for which it is enough to show
\beq{goodColorConcLargen}
\text{w.f.h.p. } |X - \E'' X| = o(D).
\enq
[The rationale for ``enough to show" follows that for \eqref{goodColorConcsmalln}
(with each $\E$ replaced $\E''$):  the reason for \eqref{fewLeftovers} is unchanged,
and
we add \eqref{equivProbs} to \eqref{expectTinNbd} in justifying
\eqref{correctExpect}.
(``W.v.h.p." is still okay in \eqref{fewLeftovers}
and the $\E''$-counterpart of \eqref{correctExpect},
but not in \eqref{goodColorConcLargen} and \eqref{TNv'}.)]

We again have (\ref{lipschitzForWellColoreds})
(with the same justification); to repeat:
\beq{lipschitzAgain}
X(\tau, \xi) - X(\tau', \xi') \begin{cases}
\leq C \\
\leq 0 \ \text{ if } i > D \text{ and } i \not\sim \tau^{-1}(\beta) \cap Z
\end{cases}
\enq
whenever $\tau_i' = \beta$
and $(\tau_j, \xi_j) = (\tau_j', \xi_j') $ for $j \neq i$.

Parameters for Lemma~\ref{martConc} (A) 
are also as before, except $S_i$ changes to $L_i^0$;
%$\Omega_i$ changes from $S_i \times \{0,1\}$ to $L_i^0 \times \{0,1\}$; 
thus:
\begin{itemize}
\item $\Omega_i = L_i^0 \times \{0,1\}$, $\omega_i = (\tau_i, \xi_i) \ (\forall i)$;
\item for $i \in [D]$, $\Sigma_i = \Omega_i$;
\item for $i > D$, $\Sigma_i = \Sigma_i(\tau_{[D]})
= \tau(N_i \cap N_v) \times \{0,1\}$; and
\item $c = C$.
\end{itemize}
The biggest (if still minor) change is that
the $|\tau(N_i \cap Z) \cap S_i|/(D+1)$ of (\ref{activeBound})
becomes $|\tau(N_i \cap Z) \cap L_i^0|/\ell$;
that is, for $i>D$,
\[
\pr''(\omega_i \in \Sigma_i) = |\tau(N_i \cap Z) \cap L_i^0|/\ell,
\]
so the sum in (\ref{unlikelyActive})
is now at most
\[
D + \ell^{-1} \sum_{i > D} |\tau(N_i \cap Z) \cap L_i^0| < 3D =: M,
\]
with the inequality given by (\ref{fewBystanders}).
Thus Lemma~\ref{martConc} again gives
\[
\pr''(|X - \E X| > \lambda) = \exp[-\Omega(\lambda^2 / (C^2 D))],
\]
so $|X - \E X| = o(D)$
w.f.h.p. (and more, but not quite ``w.v.h.p." since
$\omega(C (D \log n)^{1/2}) =\lambda = o(D)$ may be impossible when $n$ is large.) 

\qed

\nin \textit{Proof of \eqref{fraternalGetSlack}
for small and large $n$.} There is nothing new
here, so we keep it brief.

Fix $v$ and let
\[
X = |\{(uw,\gamma) \in F_v : E_{uw,\gamma}^v \text{ holds}\}|.
\]
For small $n$, \eqref{expectManyFraternalEvents} says 
\[
\E X \gtrsim e^{-3} \vt D/2,
\]
and if $n$ is large then \eqref{expectManyFraternalEvents}
and \eqref{summablySimilarFraternals} give
\[
\E'' X \gtrsim e^{-3} \vt D/2. 
\]
It is thus enough to show
\beq{fraternalConc}
\begin{cases}
|X - \E X| = o(D) &\text{ w.v.h.p. } \ \text{ if } n \text{ is small,} \\
|X-\E''X|  = o(D) &\text{ w.f.h.p. } \ \text{ if } n \text{ is large.}
\end{cases}
\enq

The arguments here are nearly identical to those
for \eqref{goodColorConcsmalln} and \eqref{activeBound} above, the only difference being
that $C$ ($=c$) is replaced by $1$.
(In words: changing a single $(\tau_i, \xi_i)$ can't change
the present $X$ by more than $1$.)
Thus Lemma~\ref{martConc} gives (slightly improving the earlier
bounds)
\[
\pr(|X - \E X| > \lambda) = \exp[-\Omega(\lambda^2/D)]
\]
and
\[
\pr''(|X  - \E'' X| > \lambda) = \exp[-\Omega(\lambda^2/D)],
\]
and (\ref{fraternalConc}) follows. 

\qed

\nin \textit{Proof of \eqref{alienGetSlack}
for small and large $n$.} Again there is nothing new here.

Fix $v$ and let
\[
X = |\{w \sim v : \mbox{$K_w$ holds}\} \cap Z|.
\]
(Recall $K_w$ and $Z$ were defined in \eqref{defKw}
and \eqref{defZvtau}, and note---referring to the definition of $A_v$ in \eqref{defAv}---that
$X$ is a lower bound on the l.h.s.\ of \eqref{alienGetSlack}.)
By \eqref{expectManyAlienEvents}
and (\ref{fewLeftovers})
if $n$ is small,
\[
\E X \gtrsim 2^{-1} e^{-1} \vt D;
\]
and by (\ref{expectManyAlienEvents}),
(\ref{summablySimilarAliens}), and
(\ref{fewLeftovers})
if $n$ is large,
\[
\E'' X \gtrsim 2^{-1} e^{-1} \vt D;
\]
so it is again enough to show
\beq{alienConc}
\begin{cases}
|X - \E X| = o(D) \ &\text{ w.v.h.p. if } n \text{ is small,} \\
|X-\E''X| =  o(D) \ &\text{ w.f.h.p. if } n \text{ is large.} 
\end{cases}
\enq
The arguments here are again like those for \eqref{goodColorConcsmalln} and \eqref{activeBound}, 
now with $c=C$ and Lemma~\ref{martConc} again yielding
\[
\pr(|X - \E X| > \lambda) = \exp[-\Omega(\lambda^2/(C^2 D))]
\]
and
\[
\pr''(|X - \E'' X| > \lambda) = \exp[-\Omega(\lambda^2/(C^2 D))];
\]
so appropriate choices of $\gl=o(D)$ again give \eqref{alienConc}.

\qed

\section{Clusters}\label{Clusters}

We now assume $V^*$ has been colored by $\gs$, and choose the ``real'' lists $L_v$ 
for $v\in \cup C_i$.
(Thus, as promised, we discard the earlier ``dummy'' colors.
We again note that the following analysis is valid for \emph{any} $\gs:V^*\ra\gG$.
As noted at the end of Section~\ref{Sketch}, we now \emph{occasionally}---only where 
the distinction matters---use $\DD$ for $|\gG|$ ($=D+1$) in place of the slightly incorrect $D$.)

With $V_i =V^*\cup \cup_{j\le i}C_j$, 
it is enough to show that for $i=1\ldots $ and any coloring $\gs$ of $V_{i-1}$
(note we continue to use $\gs$ for the growing coloring),
\beq{notext}
\pr(\mbox{$\gs$ does not extend to $C_i$}) =o( D/n).
\enq
From now on we use \emph{w.h.e.p.}\ (with high enough probability)
for ``with probability at least $1-o(D/n)$.''

\mn

\nin
\emph{Note on parameters.}
For the rest of our discussion, $\eps$, though fixed, may best be thought of as tending slowly to 0.  
We will use $a\ll b$ as a finitary counterpart of $a=o(b)$, meaning:  
for small enough $\eps$ we can
arrange that any use of 
$a\ll b$ promises that $a/b$ is small enough to support our arguments.
We then replace  ``$\sim$''
by ``$\approx$'';
e.g.\
$a\approx b$ means $|a-b|\ll a$.
Several of our small parameters will be explicit functions of $\gd$ that our ``finitary''
notation will treat as constants; so if $\gd'$ is one of these then $a\ll b$ and $a\ll \gd' b$
are equivalent.  

For ease of reference we list a few parameters here.  Those given explicitly in terms of $\gd$ are
\[    %\beq{parameters}
\mbox{$\vr =\gd/10$ (defined at \eqref{vrb}), \,\,  $\theta = e^{-9/\gd}$ (see \eqref{theta})
\,\,and \,$\nu_0 =\vr/2$ (see \eqref{nu0}).}
\]   %\enq
Other parameters that may be worth listing here are 
$\bbb$ (an ``absolute'' constant not depending on $\gd$, defined in \eqref{EDepsD}), and
$\gz_0$ (not a constant; see \eqref{gzeps}), the benchmark for the central 
$\gz$ which is ``small'' or ``large'' depending on how it compares to $\gz_0$.
Finally, it may be worth noting that we are now done with the sparse phase, and
the meanings of symbols recycled here (e.g.\ $\zeta, \theta,$ and $C$) are unrelated to 
their earlier ones.

We also recall from ``Usage'' in Section~\ref{Sketch} that all our assertions are for sufficiently large $D$,
and that the implied constants in
$O(\cdot)$, $\gO(\cdot)$, $\Theta(\cdot)$
do not depend on our various parameters.

\mn

Write $C$ for $C_i$, recalling that
(for convenience slightly relaxing from Lemma~\ref{LACK})
\beq{EDepsD}
\mbox{$||C| -\DD| < \bbb\eps D;~$ $~C$ has external degrees and
internal \emph{non}-degrees at most $\bbb\eps  D$,}
\enq  
with $\bbb$ fixed (e.g.\ $\bbb=7$).
We assume throughout that
\beq{Csize}
|C|=\DD+x
\enq
(so $|x|< \bbb \eps D$) and
say $C$ is \emph{large} if $|C|> \DD$
and \emph{small} otherwise.
In what follows $\vvv,\www$ are always
in $C$ and $\gb,\gc$ are always in $\gG$.
We also set (for $v\in C$)
\beq{nabla}
\nabla_v =|\nabla_G(v,V\sm C)| \,\, (< \bbb \eps D).
\enq

Let $H=\ov{G}[C]$, so
\beq{dH}
d_H(v) =\nabla_v +x,
\enq
and set  
\beq{|H|}
|H| =\gz D^2.
\enq
Note that, by \eqref{EDepsD}, 
\beq{gzeps}
\gz =O(\eps),
\enq
and that %for $x>0$, we have
\beq{smolx}
x \leq 2\gz D
\enq
(since $d_H(v)\geq x$ $\forall v\in C$ implies $\gz D^2 =|H|\geq |C|x/2$).

Let $\gz_0$ satisfy 
\beq{epszeta}
\eps \ll \gz_0 D \ll 1.
\enq
and 
say $\gz$ is \emph{small} if $\gz < \gz_0$ and \emph{large} otherwise.

\mn

We use $T_\vvv$ for the
set of still-allowed colors at $\vvv$; that is,
$T_\vvv=\{\gc\in S_\vvv: \vvv\not\sim \gs^{-1}(\gc)\}$.
Notice that 
\beq{SxTx}
|S_\vvv\sm T_\vvv|\leq \nabla_v%\in C
\enq
(with equality iff $\gs$ assigns $\nabla_v$ distinct colors from $S_v$
to the neighbors of $\vvv$ outside $C$).

A first easy observation, which 
will allow us to confine our attention to the sublists
$L_\vvv\cap T_\vvv$, is:
\beq{NKK'}
\mbox{w.h.e.p. $\,|L_\vvv\sm T_\vvv| < 0.5 \gd \log n \,\,\,\, \forall \vvv$.}
\enq

\nin
\emph{Notation.}
In what follows we will often use $\ttt $ for $\log n$.

\mn
\emph{Proof of \eqref{NKK'}.}
The cardinality in \eqref{NKK'}, say $\xi$, is hypergeometric with mean
$|S_\vvv\sm T_\vvv|(1+\gd)\log n/D = O(\eps D) \ttt/D$, so
Theorem~\ref{Cher'} gives
\beq{firsttheta}
\pr(\xi> 0.5 \gd t) < [O(\eps/\gd)]^{0.5\gd t} =o( n^{-1}).
\enq
(This assumes $\eps $ is somewhat less than $ e^{-2/\gd}$, but it will eventually
be considerably smaller.)
\qed

\mn

So
we may assume we have the inequalities in \eqref{NKK'}.
To simplify notation we then recycle, writing $L_\vvv$ and $\gd$ for what were $L_\vvv \cap T_\vvv$ 
and $0.5\gd $.  Thus with $\kkk = (1+\gd)\log n$ (with our revised $\gd$), \emph{we 
assume for the rest of our discussion that}
\beq{LxTx}
\mbox{\emph{$L_\vvv$ is chosen uniformly from} $\C{T_\vvv}{k}$}
\enq
(for all $\vvv\in C$,
these choices independent).

\nin

\mn

In each of the regimes below
we will want to show w.h.e.p. existence of certain matchings in subgraphs of the (random) 
bigraph on $C\cup \gG$ with edges $\{(\vvv,\gc):\gc\in  T_\vvv\}$;
this will be partly based on the following two lemmas.

For $B$ bipartite on $\UUU \cup \VVV$ and distinct $\gb,\gc \in \VVV,$ $B^{\gb,\gc}$ 
is the bigraph gotten from $B$
by replacing $N_B(\gb)$ and $N_B(\gc)$ by 
$N_B(\gb) \cup N_B(\gc)$ and $ N_B(\gb) \cap N_B(\gc)$ (respectively).
We refer to this (asymmetric) operation as \emph{switching $(\gb,\gc)$}.
Given a bigraph $B$ We use $\oB$ for its bipartite complement  
and $N(\vvv)$ for $N_B(\vvv)$.  (As usual $r^+ := \max\{r,0\}$.)

\begin{lemma} \label{nestineq'}
Let $F$ be bipartite on $\UUU \cup \VVV$; $\gb,\gc \in \VVV;$ and $F' = F^{\gb,\gc}.$
Let $(t_\vvv : \vvv \in \UUU)$ be natural numbers with $t_\vvv \leq d_F(\vvv)$, and
let $L$ and $L'$ be random subgraphs of $F$ and $F'$ (resp.), with $L$ gotten  by choosing neighborhoods
$N_L(\vvv)$ ($\vvv\in \UUU$) independently, each uniform from $\C{N_F(\vvv)}{t_\vvv}$,
and $L'$ gotten similarly, with $N_{F'}$ in place of $N_F$. 

Then
\[
\pr(\mbox{$L'$ admits a $\UUU$-perfect matching})\leq \pr(\mbox{$L$ admits a $\UUU$-perfect matching}) .
\]
\end{lemma}

\begin{lemma}\label{LMg'A}
Let $B$ be bipartite on $\UUU\cup Z$, $|\UUU|=\JJJ = (1\pm 1.1\theta) D$,
and
\beq{dFx}
d_B(\vvv)= \JJJ+R-r(\vvv) \leq \DD \,\,\,\,\forall \vvv\in \UUU
\enq 
(note $r(v)$ can be negative), with

\nin
{\rm (a)}  $R=0\,$ if $\gz$ is small, 

\nin
{\rm (b)}   $\theta D \ge R\gg \max\{1,\gz D\}\,$ 
if $\gz$ is large,

\nin
and  
\beq{r(x)''}
\mbox{$r^+(\vvv) =O(\eps D) \,\,\forall \vvv\in \UUU\,\,$ and
$\,\,
\sum_{\vvv\in \UUU}r^+(\vvv) \leq 2\gz D^2$.}
\enq
Fix $\vs \geq \gd/4$ and let $(t_\vvv:\vvv\in \UUU)$ be positive integers,
each at least $(1+\vs)\log n$.
For each $\vvv\in \UUU$ choose $\NNN_\vvv $ uniformly from $\C{N_B(\vvv)}{t_\vvv}$,
these choices independent,
and let 
\[
K = \{(\vvv,\gc): \vvv\in \UUU, \gc\in \NNN_\vvv\}.
\]
Then 
$K$ contains a $\UUU$-perfect matching w.h.e.p.
\end{lemma}

Lemma~\ref{nestineq'} is a variant of \cite[Lemma~7.4]{APS},
and is proved in exactly the same way as that lemma,
so we will not repeat the proof here (but see \cite{Kenney}).
Lemma~\ref{LMg'A}, a substantial extension of \cite[Lemma~7.3]{APS},
is proved in Section~\ref{Matchings}.

\mn

We return to $C$.
Notice that $\gz $ small implies $C$ is small (since $\gz > 1/(2D)$ if $C$ is large).
In this case we set $C'=C$, $\gG'=\gG$, and jump to the paragraph 
following Lemma~\ref{Lsucc}.
We suppose until then that $\gz$ is large.
%, and let $|C|=\DD+x$.  

Let $F$ be the bigraph on $C\cup \gG$ corresponding to the $T_v$'s
(so with edge set $\{(v,\gc): v\in C, \gc\in T_v\}$),
and $K$ the random subgraph corresponding to the $L_v$'s.
We use $d_\gc$ for the degree of $\gc\in \gG$ \emph{in $F$}.

Set 
\beq{vrb}
\vr = \gd/10, \,\,\, b= D/(1+\vr)
\enq 
and
\beq{pee.def}
\pee =\{\gc: d_\gc> b\},
\enq
noting that (using $|C|< (1+\bbb \eps )D$ and $\eps \ll \vr$)
\beq{Psmall}
|\pee|<  |C|\DD/b < (1+2\vr)D.
\enq
Set $\U=\gG\sm \pee$, writing
$d_\U(v) $ for $ |N_F(v)\cap \U|$,
and let 
\beq{theta}
\theta = e^{-9/\gd}.
\enq

\nin
Our treatment of $C$ depends on whether
\beq{R1}
\sum\{|T_\vvv\cap T_w\cap \pee|:\vvv\www\in H\}> \theta \gz D^3.
\enq

\nin
Situations where \eqref{R1} holds 
%(which in \cite{APS} were the only possibility for large $\gz$) 
will be treated below, and until further notice we assume
it does not.
%, putting us in either R2 or R3.

\mn

From this point, mainly to emphasize that the discussion applies beyond the present situation,
\emph{we use $X$ and $Y$ for 
$C$ and $\gG$,}
continuing to use $\vvv,\www$ (resp.\ $\gb, \gc$)
for members of $X$ (resp.\ $Y$). 
We show (proving \eqref{notext}) 
\beq{K'contains}
\mbox{w.h.e.p.\ $K$ contains an $X$-perfect matching.}
\enq

\nin

In a couple places below it will be helpful to assume 
\beq{YOD}
y:= |Y|=O(D),
\enq
which we can arrange using Lemma~\ref{nestineq'}:  
Starting from a given $F$, 
we apply the lemma repeatedly, but only with $(\gb,\gc)$ satisfying
\beq{legal}
|N_F(\gb)\cup N_F(\gc)| <b.
\enq
Note that under this restriction switching $(\gb,\gc)$
has no effect on $\pee $ and $\U$, since $\gb$ and $\gc$ 
are in $\U$ both before and after the switch.

So we may assume $F$ is invariant under such switches, implying in particular
that the neighborhoods of the $\gc$'s belonging to
\[
Y^* := \{\gc\in Y: 1\leq d_\gc<b/2\}
\]
are nested. 
There is thus some $\vvv$ with $N_F(\vvv)\supseteq Y^*$ 
(since if $N_F(\gc)$ is minimal over $\gc\in Y^*$, then any $\vvv\sim \gc$
is adjacent to all of $Y^*$), so in particular $|Y^*|\leq D$.
Note also that, since removing $\gc$'s with $d_\gc=0$
affects nothing, we may assume $d_\gc>0$ $\forall \gc\in Y$.
So (using $\sum d_\gc =\sum d_\vvv\approx D^2$)
\[
|Y|\leq |Y^*| + D^2\cdot 2/b =O(D)
\]
and we have \eqref{YOD}.

Let
\beq{S}
S=\{v: d_\U(v)> \theta D/2\},
\enq
$R=X\sm S$, $|S|=s$ and $|R|=r$.
Our treatment now depends
on whether
\beq{Ssize}
s > \sqrt{\theta\gz}D.
\enq
What we really have in mind here is dividing the argument according to a counterpart of 
\eqref{R1}, \emph{viz.}
%according to whether
\beq{R2}
\sum\{|T_\vvv\cap T_w\cap \U|:\vvv\www\in H\}|> \theta \gz D^3.
%|\{(xy,\gc): xy\in H, \gc\in T_x\cap T_y\cap \U\}|> \theta \gz D^3;
\enq
But assuming \eqref{Ssize} in place of \eqref{R2} is more convenient 
and covers a little more ground:
\begin{claim}\label{CR2}
If \eqref{R2} holds, then so does \eqref{Ssize}.
\end{claim}

\nin
\emph{Proof.}
Say $\vvv\www\in H$ is \emph{good} if 
\[
|T_\vvv\cap T_w\cap \U| > \theta D/2
\]
(and \emph{bad} otherwise), and notice that
all good edges are contained in $S$.
On the other hand, since 
bad edges contribute at most $(\theta D/2)\gz D^2 = \theta \gz D^3/2$ to the l.h.s.\ of \eqref{R2},
good edges must contribute \emph{at least} $\theta\gz D^3/2$, and we have
\[
\C{s}{2} \geq |\{\mbox{good edges}\}| > \theta \gz D^2/2.
\]
\qed

\mn
\emph{Proof of \eqref{K'contains} when \eqref{Ssize} holds (and \eqref{R1} does not).}
We aim for Hall's condition, which for convenience we very 
slightly strengthen, replacing ``$\geq$'' by ``$>$'': 
\begin{claim}\label{just.match}
W.h.e.p.
\beq{R2R3toshow}
|N_{K} (A)|> |A|
\enq
for all $\0\neq A\sub X$.
\end{claim}

\mn
\emph{Proof.} 
In what follows we always take $|A| =a = D/(1+\nu)$.
Let 
\beq{nu0}
\nu_0= \vr/2,
\enq
and call
$A$ \emph{large} if $\nu< \nu_0$ and \emph{small} otherwise.

Claim~\ref{just.match} is more interesting for large $A$'s and we begin with these.
%The easy case of small $A$ will be handled at ``Small $A$'' below, and until
%then we assume $A$ is large.
A difficulty here is that we cannot afford a naive union bound
$\sum\{\pr(N_{K}(A)\sub B):B\in \C{Y}{a}\}$ with a single worst case bound on the summands,
whereas understanding possibilities for the \emph{list} of summands seems not so easy.
So the following argument takes a different approach.

\mn
\emph{Large A}:
Given 
$a= D/(1+\nu)> D/(1+\nu_0)$ and $A$ of size $a$, let $d_\gc=|N_F(\gc)\cap A|$,
$\xi_\gc=\mbone_{\{\gc\not\sim_{_{K}}A\}}$
and $\xi=\xi^A=\sum \xi_\gc =y-|N_{K}(A)|$ (where, as in \eqref{YOD}, $y=|Y|$).
So 
\beq{PNK'}
\pr(|N_{K}(A)|\leq |A|) =\pr(\xi\geq y-a) .
\enq

Recalling that $S, R, s $ and $r$ were defined in and following \eqref{S},
we classify $A$'s of size $a$ according to 
$a_1:=|A\cap R|$ and $a_2:=|A\cap S|=a-a_1$, setting
$r-a_1=p$ and $s-a_2=q$.  Denoting this by $A\sim (a_1,a_2)$, we will show,
for each $(a_1,a_2)$,
\beq{Aa1a2}
\sum_{A\sim (a_1,a_2)} \pr(\xi^A\geq y-a) = n^{-\gO_{\gd}(s)},
\enq
which (with \eqref{PNK'} and \eqref{Ssize})
gives the large $A$ part of Claim~\ref{just.match}, with plenty of room.

\mn

For the proof of \eqref{Aa1a2} we set
\beq{epsgc}
p_\gc=\pr(\xi_\gc=1) = (1-\kkk/D)^{d_\gc}
\enq
and observe that, for any $\vs>0$,
\begin{eqnarray}
\pr(\xi \geq y-a) &\le &e^{-\vs  (y-a)} \E e^{\vs\xi}\nonumber\\
&\le & e^{-\vs  (y-a)} \prod (1+p_\gc(e^\vs-1)),\label{primitive},
\end{eqnarray}
where \eqref{primitive} holds because the $\xi_\gc$'s are
\emph{negatively associated} (NA).  [See e.g.\ \cite{Pemantle} for the definition.  
Here we are using a few easy facts:
(i) for a fixed $v$ the indicators $\xi_{v,\gc} = \mbone_{\{\gc\not\sim_{_{K}}v\}}$
are NA
(this is pretty trivial but see \cite{BJ} for a stronger, more interesting statement); 
(ii) if each of the independent collections $\{\xi_{v,\gc}\}_\gc$ is NA, 
then $\{\xi_{v,\gc}\}_{v,\gc}$ 
is NA (see \cite[Prop.\ 7.1]{DR}); (iii) if $\{\xi_{v,\gc}\}_{v,\gc}$ is NA 
and the $\xi_\gc$'s
are nondecreasing functions of disjoint sets of $\xi_{v,\gc}$'s,
then the $\xi_\gc$'s are also NA (see \cite[Prop.\ 7.2]{DR}); 
(iv) if the functions $f_\gc$ are nonincreasing and the r.v.s $\xi_\gc$ are NA, then 
$\E \prod f_\gc(\xi_\gc)\leq \prod \E f_\gc(\xi_\gc)$ (see \cite[Lemma~2]{DR}).]

\mn

Fix $A\sim (a_1,a_2)$.  
Immediate properties of the $d_\gc$'s are:
\beq{Con1''}
d_\gc\leq a \,\,\,\forall \gc;
\enq
\beq{Con2''}
\sum d_\gc =|\nabla_F(A)|\geq a \DD-2\gz D^2;
\enq
and 
\beq{Con3''}
\sum_{\gc\in\U}d_\gc =|\nabla_F(A,\U)| \geq a_2\theta D/2.
\enq

Let
\[
\mbox{$\sfa = \frac{a\DD-2\zeta D^2 - a_2 \theta D/2}{a}$,
$\,\,\sfb = \frac{a_2 \theta D}{2b}$, $\,\,\sfc = y - (\sfa+\sfb)$,}
\]
and $\vt =1-\kkk/D$ (so $p_\gc = \vt^{d_\gc}$).

\begin{claim}\label{Yclaim}
The product in \eqref{primitive} is at most
\beq{EgY}
%\prod (1+ \vartheta^{d_\gamma}(e^\varsigma -1)) \leq 
e^{\varsigma \sfc} 
(1+\vartheta^b(e^\varsigma -1))^{\sfb} (1+\vartheta^a(e^\varsigma-1))^{\sfa}.
\enq
\end{claim}

\begin{proof}
We apply Proposition~\ref{Pcvx}
with
\[
\mathbb{P}(X = k) = y^{-1}|\{\gamma : d_\gamma = k\}| \,\,\, (k\in \{0\dots a\}),
\]
\[
g(x)= \log(1 + \vartheta^x (e^\varsigma -1)), 
\]
\[
\beta = a_2 \theta D/(2y), 
\]
and
\[
\alpha = (a\DD - 2 \zeta D^2)/y.
\]

\nin
[Then
$E[X] \geq \alpha$ and $\E'[X] \geq \beta$ are given by 
\eqref{Con2''} and \eqref{Con3''},
and, with $\LLL=e^\vs-1$,
convexity of $g$ follows from
\[
(1+ \TTT\vt^x)(1+ \TTT\vt^{x'})\geq (1+ \TTT\vt^{(x+x')/2})^2
\]
(equivalently, $\vt^x+\vt^{x'}\geq 2\vt^{(x+x')/2}$).]

But then with $\YYY$, as in Proposition~\ref{Pcvx}, given by
\[
\mbox{$\pr(\YYY=b) = \gb/b= \sfb/y$ and $\,\,\pr(\YYY=a) = (\ga-\gb)/a= \sfa/y$} 
\]
(and $\pr(\YYY=0)  =\sfc/y$),
the product in \eqref{primitive} is $\exp[y\E g(X)] \leq \exp[y\E g(\YYY)]$, which is \eqref{EgY}.

\end{proof}

Let $y'=y-\sfc$ ($=\sfa+\sfb$).  By Claim~\ref{Yclaim}
the bound in \eqref{primitive} is at most 
\beq{y'bd}
e^{-\vs(y-a)}e^{\varsigma \sfc} 
(1+\vartheta^b(e^\varsigma -1))^{\sfb} (1+\vartheta^a(e^\varsigma-1))^{\sfa}
~=~e^{-\vs(y'-a)}
(1+\vartheta^b(e^\varsigma -1))^{\sfb} (1+\vartheta^a(e^\varsigma-1))^{\sfa};
\enq
so for \eqref{K'contains} we need to bound the r.h.s.\ of \eqref{y'bd}.

We have
\[
\sfb = (1+\vr)\theta a_2/2
\]
and
\[
\sfa ~=~ (a\DD -2\gz D^2 -\theta a_2D/2)/a 
~=~\DD -(1+\nu)[2\gz D + \theta a_2/2],
\]
so
\begin{eqnarray}
y'-a &=&  \DD -(1+\nu)[2\gz D + \theta a_2/2] + (1+\vr)\theta a_2/2 -a\nonumber\\
&=& p+q -(1+\nu)2\gz D + (\vr-\nu)\theta a_2/2 - x\label{pqx}\\
&=:& p +\vr' s.\nonumber
\end{eqnarray}

\nin
(Recall $p$ and $q$ were defined before \eqref{Aa1a2} and $x$ is from \eqref{Csize}.)
Here $\vr' $ is at least about $\theta \vr/4$, as follows from
\[
q+(\vr-\nu)\theta a_2/2~\geq ~q+ \vr \theta a_2/4~\geq \vr\theta s/4
~\gg ~\gz D,
\]
where the inequalities use (respectively): $\nu < \nu_0=\vr/2$ (see \eqref{nu0});
$s= q+a_2$;
and \eqref{Ssize}. 
(So, since $x \leq 2\gz D$ (see \eqref{smolx}) the subtracted terms in \eqref{pqx}
are negligible.)

\mn

Since 
$\vt^b = (1-k/D)^b \leq n^{-(1+\gd)/(1+\vr)} $ (and $y' \le y=O(D) $; see \eqref{YOD}), we have
\beq{mu'}
\mu' := \sfb \vt^b +\sfa\vt^a \leq  y'\vt^b < n^{-\gd'},
\enq
with $\gd' \approx \gd-\vr$.
So, setting $\vs = \log[(y'-a)/\mu']$ 
and using \eqref{mu'} at \eqref{setgz2}, we find that the r.h.s.\ of \eqref{y'bd} is less than
\begin{eqnarray}
\exp[-\vs(y'-a)+\mu'(e^\vs -1)]   
&<& \exp[-\vs(y'-a)+\mu' e^\vs ]   \nonumber\\
&=& \exp[-(y'-a)\{\log\tfrac{y'-a}{\mu'}-1\}]\nonumber\\%\label{setgz}\\
&<& \exp[-(p+\vr's)\{\log (n^{\gd'}(p+\vr' s))-1\}].\label{setgz2}
\end{eqnarray}

So, now letting $A$ vary and using  
$
\C{r}{p}\C{s}{q}
< (eD/p)^p2^s,
$
we find that the sum in \eqref{Aa1a2} is less than
\beq{largebd}
\exp[p\log(eD/p) +s -(p+\vr' s)\log (n^{\gd'}(p+\vr' s))].
\enq
To see that this gives \eqref{Aa1a2}, just notice that 
\beq{slogn}
s(1 -\vr'\log (n^{\gd'}(p+\vr' s))) = -\gO_{\gd}(s\log n),
\enq
while
\[
p[\log(eD/p) -\log (n^{\gd'}(p+\vr' s))]
\]
is small compared to the $s\log n$ of \eqref{slogn} if $p\ll s$, and otherwise is negative 
%(less than about $-\gd' p\log D$) 
since $s>\sqrt{\theta\gz}D $.

\mn
\emph{Small A}:
For any $A$ and $B$, both of size $a\leq D/(1+\nu_0)$,  
\[
\pr(N_{K}(A)\sub B) < \prod_{\vvv\in A} \left(\frac{a}{(1-\bbb\eps )D}\right)^\kkk =
\exp\left[-ak\log\left(\frac{(1-\bbb\eps )D}{a}\right)\right]
\]
(using $d_F(v)\geq (1-\bbb\eps )D$ $\forall v$ for the first inequality).
Since $|Y|=O(D)$ (see \eqref{YOD}),
the probability of violating \eqref{R2R3toshow} with an $A$ of size $a$ is thus at most 
\beq{smallbd}
\C{D}{a}\C{O(D)}{a}\exp\left[-a\kkk\log\left(\frac{(1-\bbb\eps )D}{a}\right)\right]
<\exp\left[a \left\{2\log\left(\frac{ O(D)}{a}\right) - 
\kkk \log\left(\frac{(1-\bbb\eps )D}{a}\right)\right\}\right],
\enq
which is tiny:
for $a \leq \eps D,$ the r.h.s.\ of \eqref{smallbd} is
less than
\[
\exp\left[a \left\{ 2 \log(O(D)) - \kkk \log(\gO(1/\eps)) \right\}\right] ,
\]
which is dominated by $a \kkk \log(\gO(1/\eps))\gg a\log n$;
and for $\eps D < a < D/(1+\nu_0)$, it is less than
\[
\exp \left[a \left\{2 \log[O\left(1/\eps\right)]
- \kkk \log [(1+\nu_0)(1-O(\eps))]\right\} \right] \,<\,
\exp \left[-\gO(\eps D\cdot \vr \log n)\right] 
\,=\,n^{-\omega(1)}.
\]\qed

This completes the proof of \eqref{K'contains} for $C$ violating \eqref{R1} and
satisfying \eqref{Ssize}.\qed

%\mn  We now assume \eqref{Ssize} does not hold, so according to Claim~\ref{CR2}

\mn
%\emph{Proof of \eqref{K'contains} when \eqref{Ssize} holds (and \eqref{R1} does not).}
\emph{Proof of \eqref{K'contains} when neither of \eqref{R1}, \eqref{Ssize} holds.}
Here, according to Claim~\ref{CR2}, we
also have failure of \eqref{R2}; thus we are assuming
\beq{Tbd}
s \le \sqrt{\theta\gz}D
\enq
and 
\beq{notR2}
\sum\{|T_\vvv\cap T_\www\cap \U|:\vvv\www\in H\}|\leq \theta \gz D^3.
\enq
These have the following easy consequences.

First, \eqref{Tbd} (with \eqref{S}) implies
$|\nabla_{F}(\U)|\leq |C|\theta D/2 +\sqrt{\theta\gz} D^2$, 
which, since
\[
|\nabla_{F}(\pee)|+|\nabla_{F}(\U)| =|F|\geq
|C|\DD -\nabla_G(C, V\sm C)\geq |C|(1-\bbb \eps )D,
\]
gives
\[   
|C| |\pee| \geq |\nabla_{F}(\pee)|
\geq |C|(1-\bbb \eps -\theta/2)D -\sqrt{\theta \gz}D^2 > (1-\theta) |C|D
\]   
and
\beq{pee}
|\pee|>(1-\theta) D.
\enq

Second, combining \eqref{notR2} and the failure of \eqref{R1} gives
$\sum_{\vvv\www\in H}|T_\vvv\cap T_\www|< 2\theta \gz D^3$;
%%%%%%%%%%%%%%
\iffalse
or (equivalently)
\[
|H|^{-1}\sum_{\vvv\www\in H}|T_\vvv\cap T_\www|< 2\theta D;
\]
\fi
%%%%%%%%%%%%%%%%%
so for any $L>1$, we have
$    %\[
|\{\vvv\www\in H: |T_\vvv\cap T_\www|\geq 2L\theta D\}| < \gz D^2/L
$   %\]
and 
\beq{xyHT}
|\{\vvv\www\in H: |T_\vvv\cap T_\www|<2L\theta D\}| >(1-1/L) \gz D^2.
\enq

\mn

Let $L=1/(7\theta)$ and
\beq{J}
\III=\{\vvv: d_\U(\vvv)> D/3\} \,\,\,\, (\sub S).
\enq
\begin{claim}\label{Jclaim}
Each $\vvv\www$ in \eqref{xyHT} has at least one end in $\III$.
\end{claim}
\nin
\emph{Proof.}  Using $|\pee|\leq  (1+2\vr)D$ (see \eqref{Psmall}) and 
$|T_\vvv|\ge \DD -\nv> (1-\bbb\eps) D $ (see \eqref{EDepsD}, 
recalling that $\nv$ was defined in \eqref{nabla}), we have,
for each $\vvv\www$ in \eqref{xyHT},
\begin{eqnarray*}
2\max\{d_\U(\vvv), d_\U(\www)\} &\geq & |(T_\vvv\cup T_\www)\sm\pee|
\geq 2(1-\bbb\eps -L\theta )D - (1+2\vr)D \\
&=& (1- (2\bbb\eps +2/7 +2\vr))D > 2D/3.
\end{eqnarray*}\qed

Since $\gD_H=O(\eps D)$ (again see \eqref{EDepsD}), it follows that 
(since $\gz> \gz_0\gg \eps/D$; see \eqref{epszeta})
\beq{JgzD}
|\III|\geq (1 - 1/(7\theta))|H|/\gD_H  =\gO(\gz D/\eps) \gg \max\{\gz D,1\}.
\enq

\mn

We are now ready to randomize.  
We choose the $L_\vvv$'s (equivalently $K$) in
a few steps, with accompanying assertions.
Recall $\ttt:=\log n$ (so $\kkk = (1+\gd)\ttt$).

\mn
(I)  Choose and condition on the cardinalities
\[
|N_{K}(\vvv)\cap \U|, \,\, \vvv\in (X\sm S)\cup \III,
\]
and set 
\[
\III_0=\{\vvv\in \III: |N_{K}(\vvv)\cap \U| < k/4\}
\]
and $\III_1=\III\sm \III_0$.

\begin{obs}\label{obs.whep}
W.h.e.p.
\beq{whep2}
|N_{K}(\vvv)\cap \U| < \gd t/3 \,\,\,\, \forall \vvv\in X\sm S,
\enq
\beq{whep3}
\III_0\cap \{\vvv:|N_F(\vvv)\cap \pee|<\theta D\}=\0,
\enq
and
\beq{whep4}
|\III_0| < \max\{n^{-\gO(1)}|I|,O(1)\}.
\enq
\end{obs}

\mn
\emph{Proof.}
The cardinalities in \eqref{whep2} are
hypergeometric with means $O(\theta t)$
(since $d_\U(\vvv) \leq \theta D/2$ for $\vvv\in X\sm S$); so
(as for \eqref{NKK'}, with the $\eps$ in \eqref{firsttheta}
replaced by $\theta$; see \eqref{theta})
Theorem~\ref{Cher'} says that \eqref{whep2} holds w.h.e.p.
%(So here---cf. the proof of \eqref{NKK'}---we assume, say, 
%$\theta < e^{-4/\gd}$.)

Similarly, for $\vvv$ with $|N_F(\vvv)\cap \pee|<\theta D$, we have 
$\E |N_K(\vvv)\cap \pee|< 2\theta \kkk$
(the extra 2 allows for the possibility that $d_F(\vvv) = T_\vvv$
is a little smaller than $D$); so Theorem~\ref{Cher'} gives 
\[
\pr(\vvv\in \III_0) =\pr(|N_K(\vvv)\cap \pee| > 3\kkk/4) =o(1/n),
\]
and \eqref{whep3} holds w.h.e.p.

Finally, Theorem~\ref{T2.1} 
%(and the definition of $\III$) 
gives $\pr(\vvv\in \III_0) = n^{-\gO(1)}$ for $\vvv\in \III$; so
Theorem~\ref{Cher'} says \eqref{whep4} holds w.h.e.p.
[In more detail:  With $n^{-\ga}$ the upper bound on $\pr(\vvv\in \III_0)$,
let $a=\max\{n^{-\ga/2}|I|, 4/\ga\}$.  Then $\E |\III_0|\leq n^{-\ga}|I|$ and 
Theorem~\ref{Cher'} gives 
\[
\pr(|\III_0|>a) < (e n^{-\ga}|I|/a)^a \leq (en^{-\ga/2})^{4/\ga} = o(n^{-1}).]
\]
\qed

So we may assume the choices in (I) satisfy \eqref{whep2}-\eqref{whep4}.

\mn
(II)  Choose the sets $L_\vvv':=N_{K}(\vvv)\cap \pee$ for $\vvv\in \III_0$.
%, again recalling that $x\in \III_0$ means $|M_x'|> 3\kkk/4$.

\mn
\begin{claim}  \label{P0claim}
W.h.e.p.\ $K[\III_0,\pee]$ contains an $\III_0$-perfect matching.
\end{claim}

\nin
\emph{Proof.}
We again aim for the version of Hall's condition in \eqref{R2R3toshow}.
The probability that this is violated by some $A\sub \III_0$ of size $a$---that is, that there 
are $A\in \C{\III_0}{a}$ and $B\in \C{\pee}{a}$ with $N_{K}(A)\sub B$---is at most
\[
\C{|\III_0|}{a}\C{|\pee|}{a}\left(\frac{a}{\theta D}\right)^{3a\kkk/4} < 
D^{2a} \left(\max\{n^{-\gO(1)}, O((\theta D)^{-1})\}\right)^{3a\kkk/4} 
<n^{-\go(1)},
\]
where the $\theta D$ is given by \eqref{whep3}; $3\kkk/4$ is the lower bound on
$|N_{K}(\vvv)\cap \pee|$ for $\vvv\in \III_0$; 
and the first inequality uses $a\leq |\III_0|$ and \eqref{whep4}
(and, superfluously, $|I|\leq s \ll \theta D$; see \eqref{Tbd}).
The claim follows.

\qed

Let $\pee_0$ be the set of vertices of $\pee$ used in the matching of Claim~\ref{P0claim},
and $\pee_1=\pee\sm \pee_0$.

\mn
(III)  Choose $L_\vvv$'s for $X\sm \III$.  Here we want
to say that w.h.e.p.\ $K[X\sm \III,Y\sm \pee_0]$ contains an 
$(X\sm \III)$-p.m.\ \emph{using few vertices of} $\U$.

We first observe that w.h.e.p.
\beq{whep5}
|N_K(\vvv)\cap \pee_0|\leq \gd \ttt/3 \,\,\,\,\forall \vvv.
\enq
\emph{Proof.}
This is again like \eqref{NKK'}.  Since $|\pee_0|=|\III_0|$,
\eqref{whep4} gives (for any $\vvv$)
\[
\E |N_K(\vvv)\cap \pee_0| \leq 2\kkk|\pee_0|/D < \kkk \max\{n^{-\gO(1)},O(1/D)\},
\]
which with Theorem~\ref{Cher'} gives
$\pr(|N_K(\vvv)\cap \pee_0|> \gd \ttt/3) < e^{-\go(\log n)}$; so \eqref{whep5}
holds w.h.e.p.
\qed

\mn

So we may assume \eqref{whep5}.
We will also need the following standard observation
(see e.g.\ \cite[Ex.\ 1.4.3]{LP}).

\mn
\begin{prop}\label{fact}
If $\gS$ 
is bipartite on $\VVV\cup \WWW$, $A\sub \VVV$, $M_0$ is an $A$-p.m.\ of $\gS[A,\WWW]$,
and $M_1$ is a $V$-p.m.\ of $\gS$, then there is a $V$-p.m.\ $M$ of $\gS$ using all 
vertices of $\WWW$ used by $M_0$ (and some others if $A\neq V$).
\end{prop}

We use this with $\VVV = X\sm \III$, $\WWW=Y\sm \pee_0$, $\gS = K[\VVV,\WWW]$, 
and $A$ an arbitrary subset of
$X\sm S$ of size $|C|-\theta D$ 
(which is less than $|X\sm S|$ by \eqref{Tbd}). To arrange that $M$ mostly avoids $\U$, we
aim for $M_0\sub K[A,\pee\sm \pee_0]$.

The w.h.e.p. existence of $M_0$ and $M_1$ will be given by Lemma~\ref{LMg'A}. 
In each case $B$ will be $F[\UUU,\ZZZ]$ for some $\UUU\sub X$ and $\ZZZ\sub Y$,
so, with $R$ TBA, we will have
\[
r(\vvv) =|\UUU|+R-|N_F(\vvv)\cap Z| = |\UUU|+R - |T_\vvv\cap Z|.
\]
Since $ X= C$ and $|C|=\DD+x$,
this will give \eqref{r(x)''} of Lemma~\ref{LMg'A} provided 
\beq{NGxZ}
|S_\vvv\cap \ZZZ|\geq |\UUU|+R -x\,\,\,\,\forall \vvv\in \UUU,
\enq
since then 
\[
|T_\vvv\cap Z|\geq |S_\vvv\cap Z|-|S_\vvv\sm T_\vvv| \geq |\UUU|+R -x-\nv
= |\UUU|+R -d_H(\vvv)
\]
gives $r(\vvv)\leq d_H(\vvv)$, and \eqref{r(x)''} 
follows from our assumptions on $d_H$ and $|H|$.

For $M_1$ we apply Lemma~\ref{LMg'A} with $B=F[X\sm \III,Y\sm \pee_0]$
(so $\UUU=X\sm \III$, $\ZZZ= Y\sm \pee_0$),
$R=|\III|-|\pee_0|$ ($=|\III_1|$), 
and $\vs =\gd/4$.
Then for the assumptions of the lemma:
since $I\sub S$, we have $|\UUU| = |C|-|\III|\geq |C|-s > (1-\theta) D$ 
(by \eqref{EDepsD} and \eqref{Tbd}, since
$\eps,\gz \ll \theta$);
$D\gg R$ (again since $I\sub S$ and we assume \eqref{Tbd});
$R\gg \max\{\gz D, 1\}$
is given by \eqref{JgzD} and \eqref{whep4};
\eqref{NGxZ} 
holds since $|S_v \cap Z| \geq \DD - |I_0| = |U| + R - x$; and
\eqref{whep5} gives $t_\vvv := |N_{K}(\vvv)\sm \pee_0| > (1+\vs)\log n$.

\nin

For $M_0$ we apply Lemma~\ref{LMg'A} with $B=F[A, \pee_1]$
(so $\UUU=A$, $\ZZZ= \pee_1$), $R=.49 \theta D$ 
and, again, $\vs=\gd/4$.  Then $|\UUU| =|C|-\theta D\geq (1 - 1.1\theta )D$
(again, since $|C|> (1-O(\eps)D$);
$(\theta D\ge)~ R\gg \max\{\gz D, 1\}$ (again since $\gz=O(\eps)\ll \theta$); 
for \eqref{NGxZ} we have, for $\vvv\in A$ ($\sub X\sm S$),
\begin{eqnarray*}
|S_\vvv\cap \pee_1|&\geq &
\DD- (d_\U(\vvv) + |S_\vvv\sm T_\vvv|+|\pee_0|) ~>~ \DD-.51\theta D\\
&=&|C|-x-.51 \theta D ~=~ |U|+\theta D -x -.51\theta D ~=~ |U|+R-x
\end{eqnarray*}
(with the second inequality given by \eqref{S}, \eqref{EDepsD} and 
the bound on $|\III_0|=|\pee_0|$ in \eqref{whep4});
and \eqref{whep2} and \eqref{whep5} give
(again, for $\vvv\in A$)
$t_\vvv :=|N_{K}(\vvv)\cap \pee_1| > (1+2\gd/3)t -|N_K(\vvv) \cap \pee_0|> (1+\vs)\log n$.

\mn

Proposition~\ref{fact} then gives an $(X\sm \III)$-perfect matching $M$, with 
\beq{whep6}
|\U_0:=\{\mbox{vertices of $\U$ used by $M$}\}|
\leq |X\sm \III|-|A| \leq |C|-|A| =\theta D.
\enq

\mn
(IV)  Choose $M_\vvv$'s for $\III_1$. 
(We're now only interested in their intersections with $\U\sm \U_0$).
The last thing we need to show is
\[
\mbox{w.h.e.p.\ $K[\III_1,\U\sm\U_0]$ contains an $\III_1$-p.m.}
\]

\nin
\emph{Proof.}
We again aim for \eqref{R2R3toshow}.
We claim that for $A\sub \III_1$ and $B\sub \U\sm \U_0$, both of size $a$,
\beq{NK'AU}
\pr(N_{K}(A)\cap (\U\sm \U_0)\sub B) < \left(\frac{a+|\U_0|}{D/3}\right)^{a \kkk/4}
< (4\theta)^{a\kkk/4} < n^{-(a/4)\log(1/\theta)}.
\enq

\mn
Here the first inequality uses
$d_\U(\vvv)\geq D/3$ for $\vvv\in \III$ (see \eqref{J})
and the fact that for $\vvv\in \III_1$, 
$N_{K}(\vvv)\cap \U$ is uniform from $\C{\U}{t_\vvv}$ 
for some $t_\vvv\geq \kkk/4$, and the second uses $|\U_0|\leq \theta D$ (see \eqref{whep6}) and
$a\leq |\III_1|\leq s \ll \theta D$ (see \eqref{Tbd}).

Multiplying \eqref{NK'AU} by the fewer than $D^{3a}$ possibilities for 
$(A,B)$ (we may assume $\U$ does not contain isolated
vertices, so $|\U|\leq D^2$) and summing over $a$ now
bounds the probability that \eqref{R2R3toshow} ever fails 
by (say) $n^{-0.2 \log (1/\theta)}$ ($=o(1/n)$).

\qed

This completes the proof of \eqref{K'contains} when  neither of \eqref{R1},
\eqref{Ssize} holds.

\qed

\mn
\emph{Proof of \eqref{notext} when \eqref{R1} does hold.}
Here, unlike in the situations above, we cannot expect a $C$-perfect matching in $K$
(or, for that matter, in $F$).
This is true even in the situation of \cite{APS} (where $S_\vvv=\gG =[\DD]$ $\forall \vvv$):
e.g.\ if $|C|>|\gG|$, or, more generally, if $|\cup_{\vvv\in C} T_\vvv| <|C|$.

Thus in this case we 
first perform a pairing step (at ``Process'' below)
that assigns common colors to some nonadjacent pairs of 
vertices of $C$, hoping to reach a situation where a matching from the remaining
vertices to the unused colors \emph{is} likely.  
This requires a few parameters, as follows.

Let $\eta = \eta_\zeta$ satisfy
\beq{etadef} \max \{ \zeta, 1/D \} \ll \eta \ll \zeta/\varepsilon \,\,(=O(1))
\enq

\nin
and set
\begin{align}
\label{qdef} q = 1 - \exp\left[-\frac{\theta \zeta k}{18 \bbb \eps}\right] \gg \eta,
\end{align}
where $\bbb$ is from \eqref{EDepsD}, and the inequality uses $\theta \zeta k/\eps \gg 
\theta \eta k = \omega(\eta)$
(see (\ref{etadef})) if $\theta \zeta k < 0.1 \bbb \eps$,
and $\eta \ll 1$ otherwise.
Let
\begin{align}
\label{defm} m = K \eta D/q,
\end{align}
with $K$ satisfying
\begin{align}
\label{Kdef} \eta \ll \eta K \ll q \ll (\zeta D/\eps) K \eta  
\end{align}
(equivalently 
$\max\{1,(\eps/(\gz D))q/\eta\}\ll K\ll q/\eta$, which is 
possible since $q \gg \eta$ and $\zeta D \gg \eps$; see \eqref{qdef} and \eqref{epszeta}).
%We could, e.g., take $K = \sqrt{q/\eta}$ if $\zeta D/\eps \geq q/\eta$, and otherwise 
%$K = \frac{q \sqrt{\eps}}{\sqrt{\zeta D} \eta}$.)
In particular,
\begin{align}\label{mllD} 
m \ll D.
\end{align}

Let 
\[    %\begin{align}
\label{hgam} H_\gamma = \{uv \in H : \gamma \in T_u \cap T_v\}.
\]    %\end{align}
Order $\mathcal{P} = \{\gamma_1, \gamma_2, \cdots\}$ so that
$|H_{\gamma_1}| \geq |H_{\gamma_2}| \geq \cdots$.
\begin{obs}\label{OHgamma}
For any $i \leq m, |H_{\gamma_i}| \geq 2\theta \zeta D^2/3.$
\end{obs}
\begin{proof}
Since $\sum_{\gamma \in \mathcal{P}} |H_\gamma| > \theta \zeta D^3$
and $|H_\gamma| \leq \zeta D^2$ $\forall \gc$,
\[
|H_{\gamma_i}| ~\geq~ |H_{\gamma_m}| 
~\geq~ \frac{\theta \zeta D^3 - (m-1)\zeta D^2}{|\mathcal{P}| -m+1} 
~\geq~ (\theta \zeta D^2 - m \zeta D)/(1+2\vr) > 2\theta\gz D^2/3,
\]
where the third inequality uses $|\mathcal{P}| < (1+2\vr)D$
(see \eqref{Psmall}) and the last uses $m \ll \theta D$ (see (\ref{mllD})).

\end{proof}

Write $d_H(z,\gamma)$ for $|\{w \in C : wz \in H_\gamma\}|$.
Noting that clusters other than $C$ play no 
role here, we now repurpose the notation $C_j$:

\mn
\textbf{Process:} Set $C_0 = C, H_0 = H$ and $N_0 = \emptyset$,
and for $i = 1,\cdots,m,$ reveal $\{v \in C : \gamma_i \in L_v\}$, set
\begin{align}
\label{defJi} J_i = \{v \in C_{i-1} : \gamma_i \in L_v\},
\end{align}

\nin
and do:

\begin{itemize}
	\item[(I)] If $H[J_i] \neq \emptyset$, choose (arbitrarily) $x_i y_i \in H[J_i]$ and set
	\begin{align*}
	\sigma(x_i) = \sigma(y_i) = \gamma_i, N_i = N_{i-1} \cup \{x_i y_i\}, 
	C_i = C_{i-1} \setminus \{x_i, y_i\}, \text{ and } H_i = H_{i-1} - \{x_i, y_i\} \,\,\,(= H[C_i])
	\end{align*}
	\item[(II)] If $H[J_i] = \emptyset$ and $J_i \neq \emptyset,$ choose $z_i \in J_i$
with $d_H(z_i)$ minimum and set
	\begin{align*}
	\sigma(z_i) = \gamma_i, N_i = N_{i-1}, C_i = C_{i-1} \setminus \{z_i\}, 
	\text{ and } H_i = H_{i-1} - z_i \,\,\,(=H[C_i]).
	\end{align*}
	\item[(III)] If $J_i = \emptyset$ do nothing. 
	(Updates are unnecessary here as we will take the process to have failed.)
\end{itemize}

Say the process \emph{succeeds} if 

\nin
(S1) $|L_v \cap \{\gamma_1, \cdots , \gamma_m\}| \leq 0.1 \delta \log n  \forall v \in C$;

\nin
(S2) $J_i\neq\0$ $\forall i\in [m]$, and for $i$ with $H[J_i] = \emptyset$,
$d_H(z_i) < 6 \zeta D/\gd$;

\nin
and 

\nin
(S3) $|N_m| \geq \eta D.$

\begin{lemma}\label{Lsucc}
The process succeeds with probability $1-o(D/n)$.
\end{lemma}

Lemma~\ref{Lsucc} is proved in Section~\ref{Process}.  For
now we assume 
the process has finished successfully (producing an updated $\gs$), and continue.
As promised earlier (see following Lemma~\ref{LMg'A}), 
we treat this in tandem with the case that $\gz$ is small, noting that 
``$\gz$ large'' now means that we are 
post-process in ``regime'' \eqref{R1}.

In either case we use $(C',\gG')$ for what's left of $(C,\gG)$; 
so $(C',\gG')$ is $(C,\gG)$ if $\gz$ is small
and $(C_m,\gG\sm\{\gc_1\dots \gc_m\})$ if $\gz$ is large.
It is enough to show 
\beq{K'C'G'}
\mbox{w.h.e.p. $K[C',\gG']$ contains a $C'$-perfect matching,}
\enq
for which we use Lemma~\ref{LMg'A}
with $B=F[C',\gG']$, $K=K[C',\gG']$
and
\[
R=\left\{\begin{array}{ll}
0&\mbox{if $\gz$ is small,}\\
N_m&\mbox{if $\gz$ is large,}
\end{array}\right.
\]
and just need to check the assumptions of the lemma:

For $J := |C'|$ we have
\[
J=\left\{\begin{array}{ll}
|C|= D\pm O(\eps D)&\mbox{if $\gz$ is small,}\\
|C| - (m+R) \geq D-(m+R)\pm O(\eps D)&\mbox{if $\gz$ is large;}
\end{array}\right.
\]
so in either case \eqref{mllD} (and $R\leq m$) give
$J\approx D$ and in particular $J = (1\pm 1.1\theta)D$.

As to assumptions on $R$, for (a) there is nothing to show, and for (b) we have,
using \eqref{mllD}, (S3) and \eqref{etadef},
\[
D\gg m\geq R\geq \eta D \gg \max \{ \zeta, 1/D \}.
\]

For the conditions \eqref{r(x)''} on $r^+(v)$,
recall from \eqref{Csize} and \eqref{nabla} that $|C|=\DD+x$ and $\nabla_v = |\nabla_G(v,V\sm C)|$, 
noting that
\beq{dHv}
d_H(v) = \nabla_v +x
\enq
and
\[
d_B(v) \left\{\begin{array}{ll}
=d_F(v) \geq \DD -\nabla_v
&\mbox{if $\gz$ is small,}\\
\geq d_F(v)-m \geq \DD -m -\nabla_v&\mbox{if $\gz$ is large.}
\end{array}\right.
\]
Thus
\[
r(v) := J+R - d_B(v) \leq
\left\{\begin{array}{ll}
\DD+x -(\DD-\nabla_v) = \nabla_v +x
&\mbox{if $\gz$ is small,}\\
\DD+x-m - (\DD -m -\nabla_v) =\nabla_v+x
&\mbox{if $\gz$ is large;}
\end{array}\right.
\]
so in either case $r^+(v) \leq \nabla_v+x = d_H(v)$
and \eqref{r(x)''} follows from 
$|H|=\gz D^2$ and $d_H(v)\leq \bbb\eps D$
$\forall v$ (see \eqref{|H|} and \eqref{EDepsD}).

\mn

This completes the proofs of \eqref{K'contains} and \eqref{notext}.

\section{Proof of Lemma~\ref{Lsucc}}\label{Process}

We first dispose of the easy (S1), for which
Theorem~\ref{Cher'} gives
\beq{S1holds}
\mathbb{P}(|L_v \cap \{\gamma_1, \cdots, \gamma_m\}| > 0.1 \delta \log n)
~<~
\left((1+\delta) em/(0.1 \delta D) \right)^{0.1 \delta \log n} 
~=~ o(1/n),
\enq
where we use $m \ll D$ (see (\ref{mllD})) for the second inequality
(with room).

\mn

Set
\begin{align}
\label{defQi} Q_i = \{ |L_v \cap \{\gamma_1, \cdots, \gamma_{i-1} \}| < 0.1 \delta \log n \ \forall v \in C\};
\end{align}
thus $Q_1 \supseteq \cdots \supseteq Q_m$ and $Q_m$ contains (S1).
Setting $p = (k-0.1 \delta \log n)/D = (1+0.9 \delta) \log n/D,$ we have 
\beq{getv} 
\mbox{$\mathbb{P}(v \in J_i | Q_i) > p\,\,$ 
for each $v \in C_{i-1}$ with $\gamma_i \in T_v$.}
\enq
(In fact, \eqref{getv} holds given \emph{any} history through step $i-1$ that satisfies $Q_i$.
The reason for the $Q_i$'s is that conditioning as in (\ref{getv}) becomes problematic
if we replace $Q_i$ by (S1).)
\begin{prop}
\label{s2prop} The probability of violating (S2) is $o(D/n)$.
\end{prop}
\begin{proof}
Let $Z = \{z \in C : d_H(z) < 6\zeta D/\delta\}$. With
\begin{align}
\label{Ei} E_i = \{J_i \cap Z = \emptyset\}
\end{align}
($i \in [m]$), we want
\begin{align}
\label{wantnoEi} \mathbb{P}(\cup E_i) = o(D/n).
\end{align}
Using \eqref{S1holds} and $Q_i\supseteq Q_m$, we have
\begin{align}
\label{conditiononQi} \mathbb{P}(\cup E_i) \leq \mathbb{P}(\ov{Q}_m) + 
\sum_i \mathbb{P}(E_i \land Q_m) < o(1/n) + \sum_i \mathbb{P}(E_i | Q_i);
\end{align}
so, since $m \ll D$, (\ref{wantnoEi}) will follow from (e.g.)
\begin{align}
\label{EgivQbd} \mathbb{P}(E_i | Q_i) < \exp[-(1+ \delta/3) \log n].
\end{align}
Since
\begin{align}
\label{handshakeineq} |C \setminus Z| \leq \frac{2 \zeta D^2}{6 \zeta D/\delta} = \delta D/3,
\end{align}
we have, with $Z_i = \{z \in Z \cap C_{i-1} : \gamma_i \in T_z\}$,
\begin{align*}
|Z_i| > b-2m -\delta D/3 > (1-\delta/2) D,
\end{align*}
where the $b$ uses $\gc_i\in \pee$ (see the line before Observation \ref{OHgamma},
and \eqref{pee.def} for the definition of $\pee$),
the $2m$ bounds the number of vertices removed in the process so far, and the $\delta D/3$ is
given by (\ref{handshakeineq}) (and we use $\vr = \delta/10$; see \eqref{vrb}).
Then for (\ref{EgivQbd}) we apply (\ref{getv}), yielding
\begin{align*}
\mathbb{P}(E_i | Q_i) < (1-p)^{|Z_i|} < \exp[-(1+\delta/3) \log n].
\end{align*}
\end{proof}

Now aiming for (S3), define (for $i \in [m]$)
\begin{align*}
R_i = \{ d_H(z_j) < 6 \zeta D/\delta \,\,\, \mbox{for every $ z_j $ with $ j < i$}\},
\end{align*}
noting that Proposition \ref{s2prop} gives
\begin{align}
\label{Riineq} \mathbb{P}(\cap R_i) = 1-o(D/n).
\end{align}
As was true for the $Q_i$'s, (see \eqref{conditiononQi}), \eqref{Riineq} will allow us to focus
on the choices of $J_i$'s under conditioning on $R_i$'s
(as well as $Q_i$'s), so that the next assertion becomes the main point.
\begin{prop} \label{reasonforq}
For each $i \in [m],$ if $|N_{i-1}| < \eta D$, then
\begin{align}
\label{probEmpty} \mathbb{P}(H[J_i] = \emptyset | Q_i R_i) < 
\exp \left[ - \frac{\theta \zeta k}{18\bbb\eps} \right] \ (= 1-q\text{; see (\ref{qdef})}).
\end{align}
\end{prop}

\nin
(This is roughly the best one can expect: it can happen that 
$H$ can be covered by some $T \subseteq C$ of size
$O(\zeta D^2/(\eps D)) = O(\zeta D/\eps)$, and the probability in
(\ref{probEmpty}) is then at least about $(1-k/D)^{|T|} = \exp[-O(\zeta k/\eps)].$)
\begin{proof}
Under conditioning on any $(L_v \cap \{\gamma_1, \cdots, \gamma_{i-1}\})_{v \in C}$
for which $Q_i R_i$ holds, we have
\begin{align*}
\mbox{$\mathbb{P}(v \in J_i ) > p\,\,$ 
for each $v \in C_{i-1}$ with $\gamma_i \in T_v$.}
\end{align*}
(see (\ref{getv})) and
\begin{align}
\nonumber |H_{i-1} \cap H_{\gamma_i}| &\geq 2\theta \zeta D^2/3 - 
[2\eta D \cdot  \bbb\eps D + m \cdot 6 \zeta D / \delta] \\
\nonumber &\geq \theta \zeta D^2/2 =: M.
\end{align}
Here the $2\theta \gz D^2/3$ is from Observation~\ref{OHgamma} and the
subtracted terms bound, respectively:
the number of edges of $H$ meeting members of $N_{i-1}$
(using $|N_{i-1}| < \eta D$), and
the number of edges of $ H$ containing $z_j$'s with
$j<i$ (using $R_i$). 
The second inequality is given by (\ref{etadef}) and (\ref{mllD}).

So it is enough to show
\begin{align*}
\mathbb{P}^*(H[J_i] = \emptyset) < \exp\left[-\frac{\theta \zeta k}{18 \bbb\eps}\right],
\end{align*}
where we write $\mathbb{P}^*$ for probabilities under the worst case assumption
\begin{align*}
\mbox{$\{|H_{i-1} \cap H_{\gamma_i}| = M\} \land \{\pr(v \in J_i) = p \,\,$
for each $v \in C_{i-1}$ with $\gamma_i \in T_v$\}.}
\end{align*}
This is an application of Theorem~\ref{TJanson}, for which we have
(say)
\begin{align*}
\mu = M p^2 > \theta \zeta k^2 /3
\end{align*}
and (since degrees in $H_{i-1}$ are less than $\bbb\eps D$; see \eqref{EDepsD})
\begin{align*}
\overline{\Delta} < 2M\bbb\eps D p^3 + \mu < 2 \bbb\eps \theta \zeta k^3.
\end{align*}
So the theorem gives
\begin{align*}
\pr^*(H[J_i] = \emptyset) \leq \exp[-\mu^2/\overline{\Delta}] < \exp \left[- \frac{\theta \zeta k}{18 \bbb\eps} \right].
\end{align*}
\end{proof}
We can now dispose of the last item in Lemma~\ref{Lsucc}, namely
\begin{align}
\label{gets3} \pr(\text{(S3)}) = 1-o(D/n).
\end{align}
Let $\psi_1, \cdots, \psi_m$ be independent Ber$(q)$ r.v.s and $\psi = \sum \psi_i$. 
We first show that (\ref{gets3}) will follow from
\begin{align}
\label{psiconc} \pr(\psi < \eta D) < o(1/n). 
\end{align}
To see this, set (for $i \in [m]$)
\begin{align*}
\xi_i = \begin{cases}
0 \text{ if } Q_i R_i \text{ fails or } |N_{i-1}| \geq \eta D, \\
\mathbbm{1}_{\{H[J_i] = \emptyset \}} \text{ otherwise,}
\end{cases}
\end{align*}
and notice that
\begin{align*}
\{ \text{(S3) fails} \} = \{|N_m| < \eta D\} \subseteq \cup \overline{Q_i R_i} \cup \left\{\sum \xi_i > m - \eta D \right\}
\end{align*}
(since $\{|N_m| < \eta D\} \cap \cap (Q_i R_i)$ implies $\sum \xi_i = m - |N_m|$).
On the other hand Proposition \ref{reasonforq} says we can couple $(\xi_i)$ and $(\psi_i)$
with $\xi_i \leq 1 - \psi_i$, whence, using \eqref{S1holds} (which gives $\pr(\ov{Q}_i) =o(1/n)$) and \eqref{Riineq},
\begin{align*}
\pr(\text{(S3) fails})
\leq \pr(\cup \overline{Q_i R_i}) + \pr(\sum \xi_i > m-\eta D) 
<o(D/n) + \pr(\psi < \eta D).
\end{align*}

\begin{proof}[Proof of (\ref{psiconc})]
We have $\E \psi = mq = K \eta D$
(see (\ref{defm})), so Theorem~\ref{T2.1} gives (say)
\begin{align*}
\pr(\psi < \eta D) < \exp[-K \eta D/3].
\end{align*}
This implies (\ref{psiconc}) if (say) $\theta \zeta k / \eps \leq 1$,
since in this case $q = \Theta(\theta \zeta k / \eps)$
and, by (\ref{Kdef}), 
\begin{align*}
K \eta D \gg \eps q / (\theta \zeta) = \Theta(k) \,\,\,( = \Theta(\log n)).
\end{align*}
When $\theta \zeta k / \eps > 1$ 
(in which case $q=\gO(1)$), we need a better (naive) bound:
setting $\vartheta = 1-q = \exp \left[- \frac{\theta \zeta k}{18 \bbb\eps} \right]$,
and using $m = K \eta D/q$,
we have
\begin{align*}
\pr(\psi < \eta D) < {m \choose \eta D} \vartheta^{m - \eta D}
< \exp[\eta D \log(eK/q) - \Theta((\theta \zeta k/\eps) K \eta D )] = o(1/n).
\end{align*}
Here the last inequality (an understatement) is again based on (\ref{Kdef}),
which in ``exp" gives: 

\nin
(a) $K \eta D (\theta \zeta / \eps) \gg q =\Omega(1),$
so the second term, say $T$, satisfies $T \gg k$; and 

\nin
(b) $T$ dominates the first term, since $K \gg 1$ (by (\ref{Kdef})), $\theta \zeta k/\eps > 1$,
and $q = \Omega(1)$.
\end{proof}

This completes the proof of Lemma~\ref{Lsucc}.

\section{Matchings}\label{Matchings}

The main purpose of this section is to prove Lemma~\ref{LMg'A}, which we recall:
\begin{lemma}\label{LMg'A'}
Let $B$ be bipartite on $\UUU\cup Z$, $|\UUU|=\JJJ= (1 \pm 1.1\theta)D$,
and
\beq{dFx}
d_B(\vvv)= \JJJ+R-r(\vvv) \le \DD \,\,\,\,\forall \vvv\in \UUU
\enq 
(note $r(v)$ can be negative), with

\nin
{\rm (a)}  $R=0\,$ if $\gz$ is small, 

\nin
{\rm (b)}   $\theta D\ge R\gg\max\{1,\gz D\}\,\,$ 
if $\gz$ is large,

\nin
and  
\beq{r(x)}
\mbox{$r^+(\vvv) =O(\eps D) \,\,\forall \vvv\in \UUU\,\,$ and
$\,\,
\sum_{\vvv\in \UUU}r^+(\vvv) \leq 2\gz D^2$.}
\enq
Fix $\vs\geq \gd/4$ and let $(t_\vvv:\vvv\in \UUU)$ be positive integers,
each at least $(1+\vs)\log n$.
For each $\vvv\in \UUU$ choose $\NNN_\vvv $ uniformly from $\C{N_B(\vvv)}{t_\vvv}$,
these choices independent,
and let 
\[
K = \{(\vvv,\gc): \vvv\in \UUU, \gc\in \NNN_\vvv\}.
\]
Then 
$K$ contains a $\UUU$-perfect matching w.h.e.p.
\end{lemma}

\nin
This will follow from Lemma~\ref{LMg.simple} with an assist from the following
statement, which is gotten by repeated application of Lemma~\ref{nestineq'}.
\begin{cor}\label{Cnesting}  
For $U,Z,(t_\vvv:\vvv\in U)$ and $L$ as in Lemma~\ref{nestineq'},
and any $(d_\vvv:\vvv\in U)$, the probability that $L\cap F$ admits an $U$-perfect matching is minimized 
over $F$ with $d_F(\vvv)=d_\vvv $ $\forall \vvv\in U$
when, for some ordering of $Z$, each $\vvv\in U$ is adjacent to the first $d_\vvv$ vertices of $Z$.
\end{cor}

\qed

\begin{lemma}\label{LMg.simple}
Let $B$ be bipartite on $\UUU\cup Z$,
with $|\UUU|=J = (1 \pm 1.1 \theta) D$; $|Z|=J+R$; 

\nin
{\rm (a)}  $R=0\,$ if $\gz$ is small,

\nin
{\rm (b)}   $\theta D\ge R\gg\max\{1,\gz D\}\,$ 
if $\gz$ is large;

\nin
and 
\beq{r(x)'}
\mbox{$d_{\oB}(\vvv) \ll D \,\, \forall \vvv\in \UUU\,$ and
$\,\,
\sum_{\vvv\in \UUU}d_{\oB}(\vvv)  \leq 2\gz D^2$.}
\enq

\nin
Fix $\vs \geq  \gd/8$ and let $(t_\vvv:\vvv\in \UUU)$ be positive integers,
each at least $(1+\vs)\log n$.
For each $\vvv\in \UUU$ choose $\NNN_\vvv $ uniformly from $\C{N_B(\vvv)}{t_\vvv}$,
these choices independent,
and let 
\[
K = \{(\vvv,\gc): \vvv\in \UUU, \gc\in \NNN_\vvv\}.
\]
Then 
$K$ contains an $\UUU$-perfect matching w.h.e.p.
\end{lemma}

\nin
(Note this differs from Lemma~\ref{LMg'A'} only in that it has no $r(v)$ (or \eqref{dFx}),
replaces \eqref{r(x)} by \eqref{r(x)'}, and slightly relaxes the lower bound on $\vs$.)

\mn

Most of this section is devoted to the proof of Lemma~\ref{LMg.simple}, which is similar to
that of \cite[Lemma~7.3]{APS}; but before turning to this we record the easy derivation of
Lemma~\ref{LMg'A} (a.k.a.\ Lemma~\ref{LMg'A'}).

\begin{proof}[Proof of Lemma~\ref{LMg'A}]
By Corollary~\ref{Cnesting}, it is enough to prove Lemma~\ref{LMg'A'}
when each $\vvv\in \UUU$ is adjacent to the first $d_\vvv$ vertices of $Z$
(for some ordering of $Y$).
Assuming this, we show a little more:  with $Z'$ comprising the first $J+R$ elements of $Z$,
\beq{KB'}
\mbox{w.h.e.p. $K[\UUU,Z']$ contains a $\UUU$-perfect matching,}
\enq

We just need two small observations to allow use of Lemma~\ref{LMg.simple}.
First, with $B'=B[\UUU,Z']$ (and $\vvv\in\UUU$),
\[
d_{B'}(\vvv) = \min\{d_{B}(\vvv),|Z'|\} \geq J+R-r^+(\vvv);
\]
that is, $d_{\ov{B'}}(\vvv) \le r^+(\vvv)$.

Second, since $|N_B(\vvv)\sm Z'| \leq \DD - (J+R) \le 1.1 \theta D$
(so $\E |\NNN_\vvv\sm Z'| =O(\theta \log n)$,
Theorem~\ref{Cher'} says 
(as in the proofs of \eqref{NKK'}, \eqref{whep2}) that w.h.e.p.
\beq{LxZ'}
|\NNN_\vvv\sm Z'| < (\vs/2)\log n \,\,\forall \vvv\in \UUU.
\enq
So we may assume \eqref{LxZ'} holds, and \eqref{KB'} 
is then given by Lemma~\ref{LMg.simple} with 
$B$ and $\vs$ replaced by $B'$ and $\vs/2$.

\end{proof}

\nin
\emph{Remark.}
The original version of Section~\ref{Clusters} did not use \eqref{NKK'} and \eqref{LxTx},
and was based on
the following variant of Lemma~\ref{nestineq'} which, though no longer
needed here, may still be of interest. 

\begin{cor} \label{graphingraph}
Let $G$ and $F \subseteq G$ be bipartite on $\UUU \cup \ZZZ$, $\gb,\gc \in \ZZZ,$ 
$G' = G^{\gb,\gc},$ 
and $F' = F^{\gb,\gc}.$

Let $(s_\vvv : \vvv \in \UUU)$ be natural numbers with $s_\vvv \leq d_G(\vvv)$ 
and let $M$ and $M'$ be
 random subgraphs of $G,G'$, with $M$ gotten by
choosing neighborhoods $N_M(\vvv)\  (\vvv \in \UUU)$ independently, each uniform 
from ${N_G(\vvv) \choose s_\vvv}$, and $M'$
gotten similarly, with $N_{G'}$ in place of $N_G.$ Let $L = M \cap F$ and $L' = M' \cap F'.$ 
Then
\beq{ingraphineq}
\pr(\mbox{$L'$ admits a $\UUU$-perfect matching})\leq \pr(\mbox{$L$ admits a $\UUU$-perfect matching}) .
\enq
\end{cor}

\begin{proof}  
The law of $(d_L(\vvv) : \vvv \in \UUU)$ is the same as that of 
$(d_{L'}(\vvv) : \vvv \in \UUU),$
 and for any possible sequence of $L$-degrees $(t_\vvv : \vvv \in \UUU)$, Lemma~\ref{nestineq'} says
\[    %\begin{align} \label{withconditioning} 
\pr(\mbox{$L'$ has a $\UUU$-p.m. }| d_{L'}(\vvv) = t_\vvv \,\forall \vvv)   \leq 
\pr(\mbox{$L$ has a $\UUU$-p.m. }| d_{L}(\vvv) = t_\vvv \,\forall \vvv);
\]     %\end{align}
so \eqref{ingraphineq} is given by law of total probability.
\end{proof}

\mn
\emph{Proof of Lemma~\ref{LMg.simple}.}
We want to show $K$ is likely to satisfy Hall's Condition, which we reformulate as:
\beq{HCA}
\QQQ\sub \UUU, |\QQQ|\leq \lceil J/2\rceil ~\Ra ~ |N_K(\QQQ)|\geq |\QQQ|
\enq
and
\beq{HCB}
\QQQ\sub Z, ~1+R\leq |\QQQ|\leq \lfloor J/2\rfloor +R~\Ra ~ |N_K(\QQQ)|\geq |\QQQ|-R.
\enq

The event in \eqref{HCA} fails at $\QQQ\in\C{\UUU}{\uuu}$ 
iff there is $Y\in\C{Z}{\uuu-1}$ with
\beq{LvY}
\NNN_\vvv\sub Y \,\,\,\,\,\,\forall \vvv\in \QQQ,  
\enq
the probability of which is less than (say)
\[
\C{J+R}{\uuu-1}\prod_{\vvv\in \QQQ} \left(\frac{\uuu-1 }{d_B(\vvv)}\right)^t
<
\C{J+R}{\uuu-1}\left(\frac{3\uuu }{2D}\right)^{t\uuu}.
\]

\nin
So the overall probability that \eqref{HCA} fails is less than 
\beq{SumJ2}
\sum_{\uuu=1}^{\lceil J/2\rceil}\C{J}{\uuu}\C{J+R}{\uuu-1}
\left(\frac{3\uuu }{2D}\right)^{t\uuu},
\enq
which is easily (and conservatively) seen to be $o(1/n)$, the summand being largest when $\uuu=1$.

For $\QQQ$'s contained in $Z$ we observe that if $K$ violates \eqref{HCB}, and 
$I=\{y\in Z: \mbox{$y$ is isolated in $K$}\}$,
then for any $a\in [0,R+1]$ either
\beq{either}
|I|\geq a\enq
or, for some $u\in [2,J/2]$,
\beq{or}
\mbox{$\exists$ $A\in \C{Z\sm I}{u+R-a+1}~$ 
with $~|N_K(A)|\leq u-1$.}
\enq
[If $A'\in \C{Z}{u+R}$ violates \eqref{HCB} (so $u\leq J/2$) and has fewer than $a$ 
isolates (so failure of \eqref{HCB} implies 
$u\geq 2$), then any $A\in \C{ A'\sm I}{u+R-a+1}$ is as in \eqref{or}.]

Here, if we are in (a), 
%(of Lemma~\ref{Lg.simple}, so with $R=0$), 
we take $a=1$
(which puts us in essentially the standard argument for a perfect matching in a random bigraph),
and if in (b) take some $a$ with
\beq{takeany}
\max\{\gz D,1\}\ll a\ll R,
\enq
and should show that each of \eqref{either}, \eqref{or} is unlikely.
Set 
\[
q = \min_\vvv t_\vvv/d_B(\vvv)  
\,\,(=\min\{ \pr(y\in \NNN_\vvv): \vvv\in \UUU, y\in N(\vvv)\}),
\]
noting that 
\[
q > (1+\vs)\log n/(J+R) > (1+\vs')\log n/D
\]
for some $\vs'\approx \vs$.  

\mn
\emph{For} \eqref{either}:  
The probability that all $y\in Y\in \C{Z}{a}$ are isolated is less than 
$
(1-q)^{|\nabla_B(\UUU,Y)|} \leq (1-q)^{Ja-|\ov{B}|},
$
so
\[   
\pr(\eqref{either}) ~\leq ~ \C{J+R}{a}(1-q)^{Ja-|\ov{B}|} < \left[(J+R)n^{-(1+\vs'')}\right]^a
=o(D/n),
\]   
with (again) $\vs''\approx \vs$.
Here the second inequality uses $|\ov{B}|\ll Ja$, which holds in (a) (where $a=1$) since 
$|\ov{B}|\le 2\gz D^2$ (see \eqref{r(x)'}) and
$\gz D\ll 1$, and in (b) since \eqref{takeany} gives
$|\ov{B}| \leq 2 \gz D^2\ll a D$.

\mn
\emph{For} \eqref{or}:
We have
\beq{Por}
\pr(\eqref{or})  < 
\Cc{J+R}{u+R-a+1}\Cc{J}{u-1}(1-q)^{(u+R-a+1)(J-u+1)-2\gz D^2} (uq)^{u+R-a+1}.
\enq
Here the first two factors bound the numbers of possibilities for $A $ 
and $Y\in \C{\UUU}{u-1}$ containing $N_K(A)$,
the third bounds $\pr(\nabla_K(\UUU\sm Y, A)=\0)$, and the fourth the probability that 
each $z\in A$ has a neighbor in $Y$.  (Note these last two events are independent
and the events $\{\mbox{$z$ has a neighbor in $Y$}\}$ are negatively correlated).

We may bound the second factor in \eqref{Por} by $(eJ/(u-1))^{u-1}$
and the third by 
\beq{thirdbd}
n^{-(1+0.9\vs')(u+R)(1-u/J)}.
\enq
[For \eqref{thirdbd} we have
\[   
(u+R-a+1)(J-u+1)-2\gz D^2 ~\approx ~(u+R)(J-u) ~\approx ~ (u+R)(1-u/J)D,
\]   
where the first ``$\approx$'' follows from: 

\nin 
(i) $u+R-a+1\approx u+R$
(for (a) this is just $u=u$, and for (b) it holds because
$a\ll R$; see \eqref{takeany}); 

\nin
(ii) $J-u=\Theta(D)$ (since $u\leq D/2$); and 

\nin
(iii) $\gz D^2\ll (u+R)(J-u)$ (by \eqref{epszeta} in (a) 
and \eqref{takeany} in (b)).]

So the product of the second and third factors (in \eqref{Por}) is less than
\beq{2bds}
\begin{array}{ll} 
D^{(u-1)} n^{-(1+0.7\vs')(u+ R)} &\mbox{if $u< 0.1\vs' J$,}\\
n^{-u/2}&\mbox{otherwise.}
\end{array}
\enq
(Recall $u\leq J/2$.)
On the other hand, the product of the first and last factors is less than
\[
\left[\tfrac{e(J+R)}{u+R-a+1}\cdot \tfrac{u(1+\vs') \log n}{J+R}\right]^{u+R-a+1}
< (3\log n)^{u+R},
\]
which is negligible relative to the either of the bounds in \eqref{2bds}
(where the comparison with the second bound uses $R<D$).
So, finally, the r.h.s.\ of \eqref{Por} is essentially at most the appropriate bound from 
\eqref{2bds}, each of which is easily $o(1/n)$.

\qed

\end{document}